\documentclass[11pt,reqno]{amsart}
\numberwithin{equation}{section}
\newtheorem{theorem}{Theorem}
\newtheorem{proposition}{Proposition}
\newtheorem{lemma}{Lemma}

\newtheorem{remark}{Remark}[section]
\newtheorem{acknowledgement}{Acknowledgement}
\numberwithin{theorem}{section}
\numberwithin{lemma}{section}
\numberwithin{proposition}{section}
\numberwithin{equation}{section}

\def\al{\aligned}
\def\eal{\endaligned}
\def\~{\widetilde}
\def\->{\overrightarrow}
\def\bar{\overline}

\newcommand{\reals}{\mathbb{R}}
\newcommand{\Div}{\textrm{div }}
\newcommand{\Curl}{\textrm{curl }}
\newcommand{\supp}{\textrm{supp }}
\newcommand{\V}{\widetilde{v}}
\newcommand{\Omegabar}{\overline{\Omega}}
\begin{document}

\tracingpages 1
\title[a priori]{\bf A Priori Bounds for the Vorticity of Axis Symmetric Solutions to the Navier-Stokes Equations}
\author{ Jennifer Burke and Qi S. Zhang}

\address{Department of Mathematics,  University of California,
Riverside, CA 92521, USA }
\date{11-11-2008}

\maketitle

\begin{abstract}We obtain a pointwise, a priori bound for the vorticity of axis symmetric
solutions to the 3 dimensional Navier-Stokes equations. The bound
is in the form of a reciprocal of a power of the distance to the
axis of symmetry. This seems to be the first general pointwise
estimate established for the axis symmetric Navier-Stokes equations.
\end{abstract}

\begin{section}{Introduction}

Recall the incompressible Navier-Stokes equations given in
Cartesian coordinates:
$$\Delta v-(v\cdot\nabla)v-\nabla p-\partial_tv=0,\ \Div v=0$$where
the velocity field is
$v=(v_1(x,t),v_2(x,t),v_3(x,t)):\reals^3\times[0,T]\rightarrow\reals^3$
and $p=p(x,t):\reals^3\times[0,T]\rightarrow\reals$ is the
pressure. When one converts the system to cylindrical coordinates
$r,\theta, z$ with $(x_1,x_2,x_3)=(r\cos\theta,r\sin\theta,z)$ and
considers only those solutions that are axis symmetric, then
solutions are restricted to ones of the form:
 \[v(x,t)=v_r(r,z,t)\->{e_r}+v_{\theta}(r,z,t)\->{e_{\theta}}+v_z(r,z,t)\->{e_z}.\]The components $v_r,v_{\theta},v_z$
 are all independent of the angle of rotation $\theta$.  Note $\->{e_r},\->{e_{\theta}},\->{e_z}$ are the basis vectors for
 $\reals^3$given by:
 $$\->{e_r}=\left(\frac{x_1}{r},\frac{x_2}{r},0\right),\ \->{e_{\theta}}=\left(\frac{-x_2}{r},\frac{x_1}{r},0\right),\ \->{e_z}=(0,0,1).$$

Much had been accomplished along the lines of axis symmetric
solutions including the long time existence and uniqueness of strong
solutions if the space region is taken to be all of $\reals^3$, the
external force, if any, as well as the initial velocity $v_0$, are
axis symmetric, and the rotational components, $f_{\theta}$ and
$v_{0,\theta}$, are equal to zero.  That is, the no swirl case is
known, and has been since the late 1960's (see O. A. Ladyzhenskaya
\cite{L}, M. R. Uchoviskii \& B. I. Yudovich \cite{UY}, and S.
Leonardi, J. Malek, J. Necas, \& M. Pokorny \cite{LMNP}). More
recent activities, in the presence of swirl, include the results of
C.-C. Chen, R. M. Strain, T.-P.Tsai, \& H.-T. Yau in \cite{CSTY1} \&
\cite{CSTY2}, where they prove a lower bound on the blow-up rate of
axis symmetric solutions. Similar to these results, more can be
found in the work by G. Koch, N. Nadirashvili, G. Seregin, \& V.
Sverak \cite{KNSS}; under natural assumptions they address the types
singularities that can occur in solutions to the Navier-Stokes
equations.  See also the work by G. Seregin \& V. Sverak \cite{SS}.
Also in the presence of swirl, there is the paper by J. Neustupa \&
M. Pokorny \cite{NP}, proving the regularity of one component
(either $v_r$ or $v_{\theta}$) implies regularity of the other
components of the solution. Also proving regularity is the work of
Q. Jiu \& Z. Xin \cite{JX} under an assumption of sufficiently small
zero dimension scaled norms. We would also like to mention the
regularity results of D. Chae \& J. Lee \cite{CL} who also prove
regularity results assuming finiteness of another certain zero
dimensional integral. Lastly we mention the results of G. Tian \& Z.
Xin \cite{TX}, who constructed a family of singular axis symmetric
solutions with singular initial data, as well as that of T. Hou \&
C. Li \cite{HL} who found a special class of global smooth
solutions. See also a recent extension: T. Hou, Z. Lei \& C. Li
\cite{HLL}.

In our paper, in essence, we prove an upper bound for the (possible)
blow up rate of the vorticity of axis symmetric solutions to the 3
dimensional Navier-Stokes equations.  We first state a well-known a
priori bound for the rotational component of the velocity; a proof
can be found in \cite{CL} Section 3 Proposition 1, for example. From
this we prove an a priori bound on $\omega_{\theta}$, the rotational
component of the curl, in regions close to the axis of symmetry,
using a Moser's Iteration argument similar to that found in the
publication  \cite{Z}, as well as methods in \cite{CSTY1}. With our
bound on $\omega_{\theta}$, we derive a bound on the remaining
components of the curl.

We state the theorem of the paper:
\begin{theorem} \label{ourtheorem} Suppose  $v$ is a smooth,
axis symmetric solution of the 3 dimensional Navier-Stokes equations
in $\reals^3\times(-T,0)$ with initial data $v_0=v(\cdot,-T)\in
L^2(\reals^3)$, and $w$ is the vorticity. Assume further,
$rv_{0,{\theta}}\in L^{\infty}(\reals^3)$ and let $0<R\leq 1$. Then,
there exist constants, $B_1$ and $B_2$, depending only on the
initial data,
such that for all $(x,t)\in P_{2,3,R}\subset\reals^3\times(-T,0)$, where\\
$ P_{2,3,R}=\left\{(x,t): 2R<\sqrt{x_1^2+x_2^2}<3R,\ -3R<x_3<3R,\ -R^2<t<0\right\}$:
\begin{itemize}
\item[(i)] $\displaystyle|\omega_{\theta}(x,t)|
\leq\frac{B_1}{(x_1^2+x_2^2)^{\frac{5}{2}}};$
 \item
[(ii)]
$\displaystyle|\omega_r(x,t)|+|\omega_z(x,t)|\leq\frac{B_2}{(x_1^2+x_2^2)^{5}}$.
\end{itemize}\end{theorem}

Let us introduce some notation. We use $x=(x_1,x_2,x_3)$ to denote
a point in $\reals^3$ for rectangular coordinates, and in the
cylindrical system we use $r=\sqrt{x_1^2+x_2^2}$,
$\theta=\tan^{-1}\frac{x_2}{x_1}$, $z=x_3$. Let $R>0$, $0<A<B$ be
constants, and define $P_{A,B,R}$ to be the region:
\[P_{A,B,R}=C_{A,B,R}\times (-R^2,0)\] where:
\[C_{A,B,R}=\{(x_1,x_2,x_3)|\ AR\leq r\leq BR,\ 0\leq\theta\leq
2\pi,\ |z|\leq BR\}\subset\reals^3,\] is the hollowed out cylinder
centered at the origin, with inner radius $AR$, outer radius $BR$,
and height extending up and down $BR$ units for a total height of
$2BR$.
\begin{remark}The constants $B_1,B_2$ in the above theorem are recorded here:
\[\al
B_1&=c\left(\|b\|_{L^{\infty}(-R^2,0;L^2(C_{1,4,R}))}^2+R\|rv_{0,\theta}\|_{L^{\infty}(\reals^3)}
\right)^{\frac{5}{2}}\left(\|\omega_{\theta}\|_{L^2(P_{1,4,R})}+\sqrt{R}\|rv_{0,\theta}\|_{L^{\infty}(\reals^3)}\right),\\
B_2&=c\left[\left(\|b\|_{L^{\infty}(-R^2,0;L^2(C_{\frac{1}{10},10,R}))}^4
+R^2\|rv_{0,\theta}\|_{L^{\infty}(\reals^3)}+R^2\right)\|\omega_{\theta}\|_{L^2(P_{\frac{1}{10},10,R})}^2\right.\\
&\hspace{.25cm}\left.+R\|b\|_{L^{\infty}(-R^2,0;L^2(C_{\frac{1}{10},10,R}))}^4+\|v\|_{L^2(P_{\frac{1}{10},10,R})}^2+
R^3\right]^{\frac{5}{2}}\left(\|\omega_r\|_{L^2(P_{\frac{1}{10},10,R})}+\|\omega_z\|_{L^2(P_{\frac{1}{10},10,R})}\right),\eal\]
where $b=(v_r,0,v_z)$ and $c$ is a generic constant.  $B_1$ and
$B_2$ depend only on the initial data, $v_0$, by standard energy
estimates. Also they can be made to be independent of the
smallness of $R$.  Actually, $B_1,B_2\rightarrow0$ when $R\rightarrow0$.\end{remark}
\begin{remark}We assume smoothness of the solution only for technical
simplicity. One can use standard methods to treat the suitable
weak solution case.
%Note, we assume singularities occur only on the
%z-axis, the axis of symmetry. This is because it is well known
%from the work of L. Caffarelli, R. Kohn, \& L. Nierenberg
%\cite{CKN} that for a suitable weak solution to axis symmetric
%Navier-Stokes equations, the singularity set has 1-D Hausdorff
%measure=0 . Singularities off of the axis of symmetry for axis
%symmetric solutions would result is a circle of singularities,
%contrary to this 1-D Hausdorff measure property. Further, via a
%shift is space time, we assume the singularity occurs at the
%origin when time $t=0$. Further,
\end{remark}
The remainder of the paper is organized as follows:
\\Section 2:
Preliminaries\\Section 3: A priori bound for
$\omega_{\theta}$\\Section 4: A priori bound for $\omega_r$ and
$\omega_z$.\end{section}
\begin{section}{Preliminaries}
Let us recall the standard conversion of the 3 dimensional axis
symmetric Navier-Stokes equations to cylindrical form, (see
\cite{CSTY1} for example):
\[\left\{\al &\left(\Delta-\frac{1}{r^2}\right)v_r-(b\cdot\nabla)v_r+\frac{v_{\theta}^2}{r}-\frac{\partial p}{\partial r}-\frac{\partial v_r}{\partial t}=0,\\
 &\left(\Delta-\frac{1}{r^2}\right)v_{\theta}-(b\cdot\nabla)v_{\theta}-\frac{v_{\theta}v_r}{r}-\frac{\partial v_{\theta}}{\partial t}=0,\\
 &\Delta v_z-(b\cdot\nabla)v_z-\frac{\partial p}{\partial z}-\frac{\partial v_z}{\partial t}=0,\\
 &\frac{1}{r}\frac{\partial (rv_r)}{\partial r}+\frac{\partial v_z}{\partial z}=0,\eal\right. \] where $b(x,t)=(v_r,0,v_z)$ and the last equation is the divergence free
 condition. Here $\Delta$ represents the cylindrical scalar Laplacian and $\nabla$ is the cylindrical gradient field which we record
here:
 \[\Delta =\frac{\partial^2}{\partial r^2}+\frac{1}{r}\frac{\partial}{\partial r}+\frac{1}{r^2}\frac{\partial ^2}{\partial \theta ^2}+\frac{\partial^2}{\partial
 z^2}\ \
 \nabla =\left(\frac{\partial }{\partial r},\frac{1}{r}\frac{\partial }{\partial \theta},\frac{\partial }{\partial z}\right).\]

Notice, the equation for $v_{\theta}$ does not depend on the
pressure. Defining $\Gamma=rv_{\theta}$, one sees that the
function $\Gamma$ satisfies:
\begin{equation}\label{Gamma/vtheta}
\Delta \Gamma -(b\cdot\nabla)\Gamma-\frac{2}{r}\frac{\partial
\Gamma}{\partial r}-\frac{\partial\Gamma}{\partial t}=0,\ \Div b=0.
\end{equation}

Also recall the vorticity field $\omega=\Curl v$ for axis
symmetric solutions:
\[\omega(x,t)=\omega_r\->{e_r}+\omega_{\theta}\->{e_{\theta}}+\omega_z\->{e_z},\\
\]
\begin{equation}\label{curlformulas}\omega_r=-\frac{\partial v_{\theta}}{\partial z},\ \omega_{\theta}=\frac{\partial v_r}{\partial z}-\frac{\partial v_z}{\partial r},\ \omega_z=\frac{\partial v_{\theta}}{\partial r}+\frac{v_{\theta}}{r}.\end{equation}
Next we record the equations of vorticity $\omega= \Curl v$, in
cylindrical form (again, see \cite{CSTY1} for example):
\[\left\{\al &\left(\Delta-\frac{1}{r^2}\right)\omega_r-(b\cdot\nabla)\omega_r+\omega_r\frac{\partial v_r}{\partial
r}+\omega_z\frac{\partial v_r}{\partial z}-\frac{\partial
\omega_r}{\partial t}=0,\\
&\left(\Delta-\frac{1}{r^2}\right)\omega_{\theta}-(b\cdot\nabla)\omega_{\theta}+2\frac{v_{\theta}}{r}\frac{\partial
v_{\theta}}{\partial
z}+\omega_{\theta}\frac{v_r}{r}-\frac{\partial
\omega_{\theta}}{\partial t}=0,\\
&\Delta\omega_{z}-(b\cdot\nabla)\omega_{z}+\omega_z\frac{\partial
v_z}{\partial z}+\omega_{r}\frac{\partial v_z}{\partial
r}-\frac{\partial \omega_{z}}{\partial t}=0.\eal\right.\] Define
$\displaystyle\Omega=\frac{\omega_{\theta}}{r}$, then we have that
$\Omega$ satisfies:
\begin{equation}\label{Omega/omegatheta}
\Delta \Omega -(b\cdot\nabla)\Omega+\frac{2}{r}\frac{\partial
\Omega}{\partial r}-\frac{\partial\Omega}{\partial
t}+\frac{2v_{\theta}}{r^2}\frac{\partial v_{\theta}}{\partial z}=0,\
\Div b=0. \end{equation} We confirm this by utilizing the fact that
$r\Omega=\omega_{\theta}$ and thus satisfies the rotational equation
for vorticity:
\[
\left(\Delta -\frac{1}{r^2}\right)(r\Omega)
-(b\cdot\nabla)(r\Omega)+\frac{2v_{\theta}}{r}\frac{\partial
v_{\theta}}{\partial z}+\frac{v_r}{r}
(r\Omega)-\frac{\partial(r\Omega)}{\partial t}=0.\] We compute
with the product rule on each term:
\[\al\Delta(r\Omega)=&r\frac{\partial^2\Omega}{\partial r^2}+3\frac{\partial\Omega}{\partial r}+\frac{\Omega}{r}+r\frac{\partial^2\Omega}{\partial z^2},\\
-\frac{1}{r^2}(r\Omega)=&-\frac{\Omega}{r},\\
\left(-b\cdot\nabla\right)(r\Omega)=&-v_r\Omega-r(b\cdot\nabla)\Omega,\\
\frac{v_r}{r}(r\Omega)=&v_r\Omega,\\
-\frac{\partial}{\partial
t}(r\Omega)=&-r\frac{\partial\Omega}{\partial t}.\eal\] We sum the
above and the inhomogeneous term,
$\displaystyle\frac{2v_{\theta}}{r}\frac{\partial
v_{\theta}}{\partial z}$, to get:
\[r\frac{\partial^2\Omega}{\partial r^2}+\frac{\partial\Omega}{\partial r}+r\frac{\partial^2\Omega}{\partial z^2}-r(b\cdot\nabla)\Omega+2\frac{\partial\Omega}{\partial r}-r\frac{\partial\Omega}{\partial t}+\frac{2v_{\theta}}{r}\frac{\partial
v_{\theta}}{\partial z}=0.\] Grouping all but the last term,
factoring out and dividing through by $r$, provides:
\[\Delta \Omega -(b\cdot\nabla)\Omega+\frac{2}{r}\frac{\partial
\Omega}{\partial r}-\frac{\partial\Omega}{\partial
t}+\frac{2v_{\theta}}{r^2}\frac{\partial v_{\theta}}{\partial
z}=0.\]
 Notice equations (\ref{Gamma/vtheta}) and
(\ref{Omega/omegatheta}) are similar except for a sign change on
one term and the addition of an inhomogeneous term in
(\ref{Omega/omegatheta}). Equation (\ref{Gamma/vtheta}) is used in
\cite{CSTY1} to provide the lower bound on the blow-up rate for
axis symmetric solutions.
%This equation (1.1) and Moser'e Iteration yields a
%point-wise bound described in Proposition A in the Appendix, which is slightly more
%general than the point-wise bound on $v_{\theta}$ referred to in the body of this
%paper, but is not quite as strong.
As we work with equation (\ref{Omega/omegatheta})
%our only assumptions on $b$ will be
%divergence free and $b(x,t)\in L^{\infty}([0,T],L^2(\reals ^3))$.
%Both follow if in fact $b(x,t)=(v_r,0,v_z)$ since by definition
%they are true of $v(x,t)=(v_r,v_{\theta},v_z)$. We also assume the
we assume the initial condition that provides for the pointwise
bound of $v_{\theta}$ that appears in \cite{CL} which we restate
below. Note, this is also implicitly stated in \cite{NP} (in Step
3.2 p. 396-397).
\begin{proposition} \label{chaeandleeprop}Suppose  $v$ is a smooth, axis symmetric solution of the 3 dimensional Navier-Stokes
equations with initial data $v_0\in L^2(\reals^3)$.  If
$rv_{0,{\theta}}\in L^p(\reals^3)$, then $rv_{\theta}\in
L^{\infty}(0,T;L^p(\reals^3))$. In particular, if $p=\infty$,
\[
|v_{\theta}(x,t)|\leq \frac{||rv_{0,
\theta}||_{L^{\infty}(\reals^3)}}{\sqrt{x_1^2+x_2^2}}.\]
\end{proposition}

We will also utilize the scaling of the Navier-Stokes equations in
conjunction with a change of variables. We recall that scaling of
the equations now; the pair $(v(x,t),p(x,t))$ is a solution to the
system, if and only if for any $k>0$ the re-scaled pair $( \V
(x,t),\widetilde{p}(x,t) )$ is also a solution, where
\[ \V(x,t)=kv(kx,k^2t),\hspace{1cm}\widetilde{p}(x,t)=k^2p(kx,k^2t).\]
Thus, if $(v,p)$ is a solution to the axis symmetric Navier-Stokes
equations for$(x,t)\in P_{1,4,k}$, then
$(\V(\widetilde{x},\widetilde{t}),\widetilde{p}(\widetilde{x},\widetilde{t}))$
is a solution to the equation in the variables
$\widetilde{x}=\frac{x}{k},\widetilde{t}=\frac{t}{k^2}$ when
$(\widetilde{x},\widetilde{t})\in P_{1,4,1}$. We note here how
certain quantities scale or change due to the above. Here D is any
domain in $\reals^3$ and $kD=\{x:x=ky,\ y\in D\}$:
\[\al
r=\sqrt{x_1^2+x_2^2}:\hspace{3.8cm}
\widetilde{r}=&\sqrt{\left(\frac{x_1}{k}\right)^2+\left(\frac{x_2}{k}\right)^2}=\frac{r}{k}\\&\\
\\
\|v(x,t)\|_{L^2(kD\times(-(kR)^2,0))}:\hspace{0.4cm}\|\~{v}(\~{x},\~{t})\|_{L^2(D\times(-R^2,0))}=&\left(\int_{-R^2}^0\int_D|\~{v}(\~{x},\~{t})|^2d\~{x}d\~{t}\right)^{\frac{1}{2}}\\
=&\left(\int_{-(kR)^2}^0\int_{kD}|kv(x,t)|^2\frac{1}{k^5}dxdt\right)^{\frac{1}{2}}\\
=&\frac{1}{k^{\frac{3}{2}}}\|v(x,t)\|_{L^2(kD\times(-(kR)^2,0))}\\&\\
b(x,t)=(v_r,0,v_z):\hspace{3.2cm}\widetilde{b}(x,t)=&(kv_r(kx,k^2t),0,kv_z(kx,k^2t))\\=&kb(kx,k^2t),\ (x,t)\in P_{1,4,k}\\ \Rightarrow &\widetilde{b}(\widetilde{x},\widetilde{t})=kb(x,t)\\
\ &\ \\
\|b(x,t)\|_{L^{\infty}(-(kR)^2,0;L^2(kD))}:\hspace{0.1cm}
\|\~{b}(\~{x},\~{t})\|_{L^{\infty}(-R^2,0;L^2(D))}=&\sup_{-R^2\leq \~{t}<0}\left(\int_{D}|\~{b}(\~{x},\~{t})|^2d\~{x}\right)^{\frac{1}{2}}\\
=&\sup_{-(kR)^2\leq t<0}\left(\int_{kD}|kb(x,t)|^2\frac{1}{k^3}dx\right)^{\frac{1}{2}}\\
=&\frac{1}{k^{\frac{1}{2}}}\|b(x,t)\|_{L^{\infty}{(-(kR)^2,0;L^2(kD))}}\\
&\\
\omega(x,t):\hspace{3.1cm}\widetilde{\omega}(x,t)=&
k^2\omega(kx,k^2t),\ (x,t)\in P_{1,4,k}\\\Rightarrow
&\widetilde{\omega}(\widetilde{x},\widetilde{t})=k^2\omega(x,t)\\
&\\
\|\omega(x,t)\|_{L^2(kD\times(-(kR)^2,0))}:\hspace{0.4cm}\|\~{\omega}(\~{x},\~{t})\|_{L^2(D\times(-R^2,0))}=&\left(\int_{-R^2}^0\int_D|\~{\omega}(\~{x},\~{t})|^2d\~{x}d\~{t}\right)^{\frac{1}{2}}\\
=&\left(\int_{-(kR)^2}^0\int_{kD}|k^2\omega(x,t)|^2\frac{1}{k^5}dxdt\right)^{\frac{1}{2}}\\
=&\frac{1}{k^{\frac{1}{2}}}\|\omega(x,t)\|_{L^2(kD\times(-(kR)^2,0))}\eal\]

\vspace{.5cm} One can show
$\widetilde{\Gamma}(\~{x},\~{t})=\widetilde{r}\widetilde{v_{\theta}}(\~{x},\~{t})$
is a solution to (\ref{Gamma/vtheta}) and
$\widetilde{\Omega}(\~{x},\~{t})=\frac{\widetilde{\omega_{\theta}}(\~{x},\~{t})}{\widetilde{r}}$
is a solution to (\ref{Omega/omegatheta}) in the variables
$(\~{x},\~{t})\in P_{1,4,1}$. We will do most of our computations on
scaled cylinders.
%Before beginning the argument, we recall a certain scaling
%invariant quantity for n=3, as found in [SS] that will be useful
%as well
%\[\frac{1}{R^3}\int_{Q(z_0,R)}|v|^2dz\]\end{section}
\end{section}
\begin{section}{A priori bound for $\omega_{\theta}$}

In this Section, and in Section 4, we are going to drop the
"tilde" notation for the sake of simplicity for a time when
computations take place over the scaled cylinders. We will then
recall that the $L^2-L^{\infty}$ bounds derived are for scaled
functions with a change of variables and we will discuss the
consequences of this in subsections labeled "re-scaling". Note,
however, because of this scaling, we must keep a close watch on
constants that involve the quantities discussed in the
preliminaries.

\vspace{.5cm}
\textbf{Proof of Theorem \ref{ourtheorem} (i):}

In the region $P_{1,4,1}$ we do our analysis on
(\ref{Omega/omegatheta}):
\[
\Delta \Omega -(b\cdot\nabla)\Omega+\frac{2}{r}\frac{\partial
\Omega}{\partial r}-\frac{\partial\Omega}{\partial
t}+\frac{2v_{\theta}}{r^2}\frac{\partial v_{\theta}}{\partial
z}=0,\ \Div b=0.\]

A flow chart for the argument to prove part (i) of Theorem \ref{ourtheorem} is as follows:\\
Energy Estimates:\\
Step 1: Use a refined
cut-off function.\\
Step 2: Estimate drift term $(b\cdot\nabla)\Omega$ using methods similar to \cite{Z}.\\
Step 3: Estimate a term involving the cut-off.\\
Step 4: Estimate the term involving the directional
derivative $\partial_r$ using a method similar to that in \cite{CSTY1}.\\
Step 5: Estimate the inhomogeneous term utilizing the bound in Proposition \ref{chaeandleeprop} (see \cite{CL}).\\
$L^2-L^{\infty}$ Estimate on Solutions to (\ref{Omega/omegatheta}) via Moser's Iteration.\\
$L^2-L^{\infty}$ Estimate on $\omega_{\theta}$ via re-scaling.

\vspace{.5cm} \textbf{Energy Estimates:}

\vspace{.5cm} \textbf{Step 1:} We use a revised cut-off function
and the equation to obtain inequality (\ref{T1T2T3T4}) below.

 Let $q\geq 1$ be a
rational number.  We note that eventually we will be applying
Moser's iteration, where at each step
$q=\left(1+\frac{2}{n}\right)^i,\ i\in\mathbb{N}$ and here $n=3$.
Let
\begin{equation}\label{Lambda}\Lambda=||v_{\theta}||_{L^{\infty}(P_{1,4,1})}\leq
||rv_{0,\theta}||_{L^{\infty}(\reals^3)}<\infty,\end{equation}
utilizing the hypothesis that $rv_{0,\theta}\in
L^{\infty}(\reals^3)$, the point-wise bound in Proposition \ref{chaeandleeprop},
and the fact that $1<\sqrt{x_1^2+x_2^2}<4$. Let
\begin{equation}\Omegabar_+(x,t)=\left\{\begin{array}{cc}
  \Omega(x,t)+\Lambda & \Omega(x,t)\geq0, \\
  \Lambda & \Omega(x,t)<0. \\
\end{array}\right.\end{equation}
Note that $\Omegabar_+\geq\Lambda$ and all derivatives of
$\Omegabar_+$ on the set where $\Omega(x,t)<0$ are equal to zero.
This function is also Lipschitz and $\Omega$ we assume to be smooth.
At interfaces boundary terms upon integration by parts will cancel
and so the calculations below can be made sense of. Direct
computation yields:\begin{equation}\label{Omegaq}\Delta
\Omegabar_+^q -(b\cdot \nabla) \Omegabar_+^q
+\frac{2}{r}\partial_r \Omegabar_+^q-\partial _t
\Omegabar_+^q=-\frac{q\Omegabar_+^{q-1}}{r^2}\frac{\partial
v_{\theta}^2}{\partial z}+q(q-1)\Omegabar_+^{q-2}|\nabla
\Omegabar_+|^2.\end{equation} Let $\frac{5}{8}\leq
\sigma_2<\sigma_1 \leq 1$. We let
\[P_{5-4\sigma_i,4\sigma_i,1}=\{(r,\theta,z)|\
(5-4\sigma_i)<r<4\sigma_i,\ 0\leq\theta\leq 2\pi,\
|z|<4\sigma_i\}\times (-\sigma_i^2,0)\] for $i=1,2$.  For convenience denote the
space portion, which is a hollowed out cylinder, as $C(\sigma_i)$
and let
\[P(\sigma_i)=P_{5-4\sigma_i,4\sigma_i,1}=C(\sigma_i)\times(-\sigma_i^2,0).\]

Choose $\psi=\phi(y)\eta(s)$ to be a refined cut-off function
satisfying:\[\supp\phi \subset C(\sigma_1);\ \phi(y)=1 \textrm{ for all } y\in
C(\sigma_2);\ \frac{|\nabla \phi|}{\phi
^{\delta}}\leq\frac{c_1}{\sigma_1-\sigma_2}\textrm{ for }\delta \in (0,1);\ 0\leq \phi\leq 1;\
\]\[\supp \eta \subset (-\sigma_1^2,0];\
\eta(s)=1,\textrm{ for all } s\in [-\sigma_2^2,0];\ |\eta '|\leq
\frac{c_2}{(\sigma_1-\sigma_2)^2};\ 0\leq \eta \leq 1.\] Let
$f=\Omegabar_+^q$ and use $f\psi ^2$ as a test function in
(\ref{Omegaq}) to get:
\[
\al \int_{P(\sigma_1)}(\Delta f - (b&\cdot\nabla)f -\partial _s f
+\frac{2}{r}\partial
_rf)f \psi ^2 dyds\\
&=
\int_{P(\sigma_1)}q(q-1)\Omegabar_+^{q-2}|\nabla \Omegabar_+|^2f \psi ^2 dyds-\int_{P(\sigma_1)}\frac{q\Omegabar_+^{q-1}}{r^2}\frac{\partial v_{\theta}^2}{\partial z}f\psi^2dyds\\
\ &=q(q-1)\int_{P(\sigma_1)}\Omegabar_+ ^{-2}|\nabla \Omegabar_+|^2f^2\psi ^2dyds-\int_{P(\sigma_1)}\frac{q\Omegabar_+^{2q-1}}{r^2}\frac{\partial v_{\theta}^2}{\partial z}\psi^2dyds\\
\ &\geq
-\int_{P(\sigma_1)}\frac{q\Omegabar_+^{2q-1}}{r^2}\frac{\partial
v_{\theta}^2}{\partial z}\psi^2dyds. \eal
\]

Integration by parts on the first term implies:
\[\al\int_{P(\sigma_1)}\nabla (f \psi ^2&)\nabla f dyds\\
&\leq \int_{P(\sigma_1)}\left(-b\cdot\nabla f(f \psi ^2)-\partial
_s f(f \psi ^2) +\frac{2}{r}\partial _rf(f \psi ^2)
+\frac{q\Omegabar_+^{2q-1}}{r^2}\frac{\partial
v_{\theta}^2}{\partial z}\psi^2\right)dyds\eal\] A manipulation
using the product rule shows:
\[\int_{P(\sigma_1)}\nabla(f \psi ^2)\nabla f dyds=\int_{P(\sigma_1)}\left(|\nabla(f \psi)|^2-|\nabla \psi|^2f^2\right)dyds.
\]
Thus,
\[\al
\int_{P(\sigma_1)}&|\nabla(f \psi)|^2dyds\\
&\leq \int_{P(\sigma_1)}\left(-b\cdot\nabla f(f \psi ^2)-\partial
_s f(f \psi ^2) +\frac{2}{r}\partial _rf(f \psi ^2)
+\frac{q\Omegabar_+^{2q-1}}{r^2}\frac{\partial
v_{\theta}^2}{\partial
z}\psi^2+|\nabla\psi|^2f^2\right)dyds.\eal\] Integration by parts
on the term involving the time derivative yields:
\[\al
\int_{P(\sigma_1)}-(\partial _s f)f \psi ^2dyds&=-\frac{1}{2}\int_{P(\sigma_1)}\partial_s(f^2)\psi^2 dyds\\
&=-\frac{1}{2}\left(\int_{C(\sigma_1)}f^2\psi^2(y,0)dy
-\int_{C(\sigma_1)}f^2\psi^2(y,-\sigma_1
^2)dy\right)\\&\hspace{2cm}+\frac{1}{2}\int_{P(\sigma_1)}\partial_s(\psi^2)f^2dyds.\eal\]
Our
cut-off functions provides $\psi^2=\left(\phi\eta\right)^2,\ \eta(0)=1,\
\eta(-\sigma_1 ^2)=0$, and $0\leq\phi\leq1$.  Thus,
\[\al
\int_{P(\sigma_1)}-(\partial _s f)f \psi ^2dyds&=-\frac{1}{2}\int_{C(\sigma_1)}f^2(y,0)\phi^2(y)dy+\int_{P(\sigma_1)}\phi^2(\eta\partial_s\eta)f^2dyds\\
&\leq-\frac{1}{2}\int_{C(\sigma_1)}f^2(y,0)\phi^2(y)dy+\int_{P(\sigma_1)}(\eta\partial_s\eta)f^2dyds.\eal\]

And so,
\begin{equation}\label{T1T2T3T4}\al\int_{P(\sigma_1)}|\nabla(f
\psi)|^2&dyds+\frac{1}{2}\int_{C(\sigma_1)}f^2(y,0)\phi^2(y)dy\\&\leq\int_{P(\sigma_1)}-b\cdot\nabla
f(f\psi^2)dyds
+\int_{P(\sigma_1)}(\eta\partial_s\eta+|\nabla\psi|^2)f^2dyds\\&\hspace{.5cm}+\int_{P(\sigma_1)}\frac{2}{r}\partial
_r
f(f\psi^2)dyds+\int_{P(\sigma_1)}\frac{q\Omegabar_+^{2q-1}}{r^2}\frac{\partial
v_{\theta}^2}{\partial
z}\psi^2dyds\\
\ &\ \\&:=T_1+T_2+T_3+T_4.\eal\end{equation}

\vspace{1cm}
\textbf{Step 2:}
To deal with $T_1$ we refer to \cite{Z} where a parabolic
equation with a similar drift term is explored.

Since $\Div b=0$,
\[
\al T_1&=\int_{P(\sigma_1)}-b\cdot(\nabla f)(f
\psi^2)dyds\\
&=\frac{1}{2}\int_{P(\sigma_1)}-b\psi^2\cdot\nabla (f^2)dyds=\frac{1}{2}\int_{P(\sigma_1)}\Div(b\psi^2)f^2dyds\\
&=\frac{1}{2}\int_{P(\sigma_1)}\Div b (\psi
f)^2dyds+\frac{1}{2}\int_{P(\sigma_1)}b\cdot\nabla (
\psi  ^2)f^2dyds\\
&= \int_{P(\sigma_1)}b\cdot(\nabla\psi)\psi f^2dyds\\
&\leq \left|\int_{P(\sigma_1)}\left(
b\psi^{1+\delta}|f|^{2-a}\right)\left(\frac{\nabla
\psi}{\psi^{\delta}}|f|^a\right)dyds\right|, \eal\] for
$0<\delta<1,\ 0<a<2$ which we introduce in order to split the
above integral using H\"{o}lder's inequality. Apply H\"{o}lder's
inequality with exponents $\frac{4}{3}$ and 4:
\[
\al
T_1&\leq\left(\int_{P(\sigma_1)}|b|^{\frac{4}{3}}\left(\psi^{1+\delta}|f|^{2-a}\right)^{\frac{4}{3}}dyds\right)^{\frac{3}{4}}\left(\int_{P(\sigma_1)}\left(\frac{|\nabla
\psi|}{\psi^{\delta}}|f|^a\right)^4dyds\right)^{\frac{1}{4}}.\eal\]
We would like $\frac{4}{3}(1+\delta)=2,\ \frac{4}{3}(2-a)=2$. This
holds if $\delta=\frac{1}{2},\ a=\frac{1}{2}$. Using properties of
the cutoff function we get:
\[
T_1\leq\left(\int_{P(\sigma_1)}|b|^{\frac{4}{3}}(f\psi)^2dyds\right)^{\frac{3}{4}}\frac{c_1
}{\sigma_1-\sigma_2}\left(\int_{P(\sigma_1)}f^2dyds\right)^{\frac{1}{4}}.\]
Next we fix $\epsilon_1>0$ and we apply Young's inequality, with
exponents $\frac{4}{3}$ and $4$: \[\al
T_1&\leq\left(\frac{4}{3}\epsilon_1\right)^{\frac{3}{4}}\left(\int_{P(\sigma_1)}|b|^{\frac{4}{3}}(f\psi)^2dyds\right)^{\frac{3}{4}}\cdot\left(\frac{4}{3}\epsilon_1\right)^{-\frac{3}{4}}\frac{c_1
}{\sigma_1-\sigma_2}\left(\int_{P(\sigma_1)}f^2dyds\right)^{\frac{1}{4}}\\
&\leq\epsilon_1\int_{P(\sigma_1)}|b|^{\frac{4}{3}}(f\psi)^2dyds+\frac{c_3\epsilon_1^{-3}}{(\sigma_1-\sigma_2)^4}\int_{P(\sigma_1)}f^2dyds.\eal\]
Thus,
\begin{equation}\label{T1}|T_1|\leq
\epsilon_1c_4K_b^{\frac{4}{3}}(C_{1,4,1})\int_{P(\sigma_1)}|\nabla(f\psi)|^2dyds+\frac{c_3\epsilon_1^{-3}}{(\sigma_1-\sigma_2)^4}\int_{P(\sigma_1)}f^2dyds
,\end{equation} where $K_b(C_{1,4,1})$ is the constant:
\[K_b(C_{1,4,1})=\|b\|_{L^{\infty}(-1,0;L^2(C_{1,4,1}))}.\]  This
last inequality holds as a result of $b=(v_r,0,v_z)\in
L^{\infty}((0,\infty),L^2(\reals^3))$, H\"{o}lder's inequality
with exponents $\frac{3}{2}$ and $3$, and the Sobolev Inequality,
noting the dimension $n=3$:
\[\al\int_{P(\sigma_1)}|b|^{\frac{4}{3}}(f\psi)^2dyds&\leq\int_{-\sigma_1
^2}^0\left(\int_{C(\sigma_1)}|b|^2dy\right)^{\frac{2}{3}}\left(\int_{C(\sigma_1)}(f\psi)^6dy\right)^{\frac{1}{3}}ds\\
&\leq c_4\sup_{-\sigma_1^2\leq s\leq
0}\left(\int_{C(\sigma_1)}|b|^2dy\right)^{\frac{2}{3}}\int_{P(\sigma_1)}|\nabla(f\psi)|^2dyds\\
&\leq
c_4K_b^\frac{4}{3}(C_{1,4,1})\int_{P(\sigma_1)}|\nabla(f\psi)|^2dyds.\eal\]

\vspace{1cm} \textbf{Step 3:} The term involving the cut-off function,
$T_2$, is standard.  We use
\[T_2=\int_{P(\sigma_1)}(\eta\partial_s\eta+|\nabla\psi|^2)f^2dyds,\] and properties of the cutoff, \[\
|\nabla \psi |^2=|\eta \nabla \phi |^2\leq \left( \frac{|\nabla
\phi |}{\phi ^{\delta }}\right) ^2\leq
\frac{c_1^2}{(\sigma_1-\sigma_2)^2}\] and \[|\eta
\partial _s \eta |\leq |\partial _s \eta |\leq \frac{c_2}{(\sigma_1-\sigma_2)^2},\] to get:\begin{equation}\label{T2}|T_2|\leq
\frac{c_5}{(\sigma_1-\sigma_2)^2}\int_{P(\sigma_1)}f^2dyds.\end{equation}

\vspace{1cm} \textbf{Step 4:} As we deal with
$T_3=\int_{P(\sigma_1)}\frac{2}{r}\partial _rf(f\psi^2)dyds$, we
note we are assuming the integration takes place away from the
singularity set of the solution to the axis symmetric Navier
Stokes equations and away from the z-axis in general. Thus all
functions are bounded and smooth and $r$ varies between two
positive constants, confirming this quantity is
 integrable.  We also utilize the cylindrical coordinates of
the axis symmetric case, and integration by parts:
\[\al T_3&=\int_{P(\sigma_1)}\frac{2}{r}\partial
_rf(f\psi^2)dyds\\&=\int_{P(\sigma_1)}\frac{1}{r}\partial_r(f^2)\psi^2rdrd\theta
dzds\\&=\int_{P(\sigma_1)}\partial_r(f^2)\psi^2drd\theta
dzds\\&=-\int_{P(\sigma_1)}\partial_r(\psi^2)f^2drd\theta
dzds\\&=-\int_{P(\sigma_1)}\frac{2}{r}\partial_r\psi (\psi
f^2)rdrd\theta dzds\\
&=-\int_{P(\sigma_1)}\frac{2}{r}\partial_r(\psi)(\psi f^2)dyds\\
&=-\int_{P(\sigma_1)}\frac{2}{r}\overrightarrow{e_r}\cdot\nabla\psi(\psi
f^2)dyds.\eal\] The Cauchy-Schwartz inequality then implies:
\[|T_3|\leq\int_{P(\sigma_1)}\frac{2}{r}|\nabla\psi|\psi f^2dyds.\]
Next we use splitting methods similar to those found in
\cite{CSTY1}; fix $\epsilon_2>0,\ m>1$ to be chosen later and
apply Young's inequality with exponents $m$ and $\frac{m}{m-1}$:
\[
\al|T_3|&\leq\int_{P(\sigma_1)}\frac{2}{r}|\nabla\psi|\psi f^2dyds\\
&=\int_{P(\sigma_1)}\left((m\epsilon_2)^{\frac{1}{m}}\frac{2}{r}(\psi
f)^{\frac{2}{m}}\right)\times\left((m\epsilon_2)^{\frac{-1}{m}}\psi^{\frac{m-2}{m}}|\nabla
\psi|f^{\frac{2(m-1)}{m}}\right)dyds\\
&\leq\epsilon_2\int_{P(\sigma_1)}\left(\frac{2}{r}\right)^{m}\psi^2f^2dyds+\frac{c_6m^{\frac{m-2}{m-1}}\epsilon_2^{\frac{-1}{m-1}}}{m-1}\int_{P(\sigma_1)}\left(\frac{|\nabla\psi|}{\psi^{\frac{2-m}{m}}}\right)^{\frac{m}{m-1}}f^2dyds.\eal\]

Properties of the cutoff yield:
\[
|T_3|\leq\epsilon_2\int_{P(\sigma_1)}\left(\frac{2}{r}\right)^{m}(f\psi)^2dyds+\frac{c_6m^{\frac{m-2}{m-1}}\epsilon_2^{\frac{-1}{m-1}}}{(m-1)(\sigma_1-\sigma_2)^{\frac{m}{m-1}}}\int_{P(\sigma_1)}f^2dyds.\]

Now consider the quantity:
\[\int_{C(\sigma_1)}\left(\frac{2}{r}\right)^{m}(f\psi)^2dy.\]
Apply H\"{o}lder's inequality with exponents $\frac{3}{2}$ and
$=3$ and the Sobolev inequality, $n=3$, then:
 \[
 \al
 \int_{C(\sigma_1)}\left(\frac{2}{r}\right)^{m}(f\psi)^2dy
 &\leq\left(\int_{C(\sigma_1)}\left(\frac{2}{r}\right)^{\frac{3m}{2}}dy\right)^{\frac{2}{3}}\times \ \left(\int_{C(\sigma_1)}(f\psi)^6dy\right)^{\frac{1}{3}}\\
 &\leq c_7\left(\int_{C(\sigma_1)}\left(\frac{2}{r}\right)^{\frac{3m}{2}}dy\right)^{\frac{2}{3}}\times\int_{P(\sigma_1)}|\nabla(f\psi)|^2dy\\
 &\leq c_{11}\int_{C(\sigma_1)}|\nabla (f\psi)|^2dy,
 \eal
 \]
 if we choose $m$ appropriately. To see this, we calculate:
\[
\al
c_7\left(\int_{C(\sigma_1)}\left(\frac{2}{r}\right)^{\frac{3m}{2}}dy\right)^{\frac{2}{3}}&=
c_8\left(\int_{-4\sigma_1}^{4\sigma_1}\int_0^{2\pi}\int_{\sigma_1}^{4\sigma_1}\frac{1}{r^{\frac{3m}{2}}}rdrd\theta
dz\right)^{\frac{2}{3}}\\&=\left(c_9\sigma_1\sigma_1^{-\frac{3m}{2}+2}\right)^{\frac{2}{3}}\hspace{1cm}\textrm{
if we choose $1<m<\frac{4}{3}$}\\&=c_{10}(\sigma_1)^{2-m}
\\
&\leq c_{11}\hspace{3.5cm}\textrm{ since $\frac{5}{8}\leq\sigma_2<\sigma_1\leq1$.}\eal
\]
Note also:
\[
\al
c_7\left(\int_{C(\sigma_1)}\left(\frac{2}{r}\right)^{\frac{3m}{2}}dy\right)^{\frac{2}{3}}&=
c_8\left(\int_{-4\sigma_1}^{4\sigma_1}\int_0^{2\pi}\int_{\sigma_1}^{4\sigma_1}\frac{1}{r}drd\theta
dz\right)^{\frac{2}{3}}\hspace{1cm}\textrm{ if
$m=\frac{4}{3}$}\\&=c_9\sigma_1^{\frac{2}{3}}\\&\leq
c_{10}\hspace{3.5cm}\textrm{ since
$\frac{5}{8}\leq\sigma_2<\sigma_1\leq1$.}\eal
\]

Thus, allowing $1<m\leq\frac{4}{3}$ yields:
\begin{equation}\label{T3}|T_3|\leq\epsilon_2c_{11}\int_{P(\sigma_1)}|\nabla(f\psi)|^2dyds+\frac{c_6m^{\frac{m-2}{m-1}}\epsilon_2^{\frac{-1}{m-1}}}{(m-1)(\sigma_1-\sigma_2)^{\frac{m}{m-1}}}\int_{P(\sigma_1)}f^2dyds.\end{equation}

\vspace{1cm} \textbf{Step 5:}
Lastly we work on the inhomogeneous
term of (\ref{Omega/omegatheta}), that is
$\displaystyle\frac{2v_{\theta}}{r^2}\frac{\partial
v_{\theta}}{\partial z}$, which produced the term $T_4$. Recall
$\Lambda=||v_{\theta}||_{L^{\infty}(P_{1,4,1})}\leq
||rv_{0,\theta}||_{L^{\infty}(\reals^3)}<\infty$ and that\\
$\Omegabar_+=\left\{\begin{array}{cc} \Omega+\Lambda & \Omega \geq
0\\\Lambda & \Omega<0\end{array}\right.$, thus
$\Omegabar_+\geq\Lambda$. Also, we have let $f=\Omegabar_+ ^q$.
Using integration by parts yields:
\[\al T_4=&\int_{P(\sigma_1)}\frac{q\Omegabar_+^{2q-1}}{r^2}\frac{\partial
v_{\theta}^2}{\partial z}\psi^2dyds\\
=&-\int_{P(\sigma_1)}\frac{\partial}{\partial z}
\left(\frac{\Omegabar_+^{2q} \psi^2}{\Omegabar_+} \right)
\frac{q}{r^2}
v_{\theta}^2  dyds\\
=& -\int_{P(\sigma_1)}\frac{\partial}{\partial z} (f\psi)^2
\frac{1}{\Omegabar_+}  \frac{q}{r^2} v_{\theta}^2  dyds +
\int_{P(\sigma_1)} (\Omegabar_+^q \psi)^2 \frac{1}{\Omegabar_+^2}
\frac{\partial \Omegabar_+}{\partial z} \frac{q}{r^2} v_{\theta}^2
dyds \\
=& -\int_{P(\sigma_1)}\frac{\partial}{\partial z} (f\psi)^2
\frac{1}{\Omegabar_+}  \frac{q}{r^2} v_{\theta}^2  dyds +
\frac{1}{2} \int_{P(\sigma_1)} \frac{1}{\Omegabar_+}
\left[\frac{\partial (\Omegabar_+^{2q}\psi^2)}{\partial z} -
\Omegabar_+^{2q} \frac{\partial \psi^2}{\partial z}\right]
\frac{1}{r^2} v_{\theta}^2 dyds
\\
=& -\int_{P(\sigma_1)}\frac{\partial}{\partial z} (f\psi)^2
\frac{1}{\Omegabar_+}  \frac{q-(1/2)}{r^2} v_{\theta}^2  dyds -
\frac{1}{2} \int_{P(\sigma_1)} \frac{1}{\Omegabar_+}
\Omegabar_+^{2q} \frac{\partial \psi^2}{\partial z} \frac{1}{r^2}
v_{\theta}^2 dyds.\eal\]

Considering $\frac{|v_{\theta}|}{\Lambda}\leq 1$, utilizing
$\Lambda\leq\Omegabar_+$, and $r=\sqrt{y_1^2+y_2^2}\geq 1$ for all
$y\in P(\sigma_1)$, we continue by fixing $\epsilon_3>0$. Apply
Young's inequality with exponents both being $2$ to get:
\begin{equation}
\label{T4} \al
|T_4|\leq&\int_{P(\sigma_1)}2q|v_{\theta}||f|\psi\left|\frac{\partial(f\psi)}{\partial
z}\right|dyds+ \frac{c_3}{\sigma_1-\sigma_2} \int_{P(\sigma_1)}
f^{2}
|v_{\theta}| dyds\\
\leq &\int_{P(\sigma_1)}
\left(\frac{2q\Lambda}{(2\epsilon_3)^{\frac{1}{2}}}f\psi\right)\times
\left((2\epsilon_3)^{\frac{1}{2}}\frac{\partial(f\psi)}{\partial
z}\right)dyds + \frac{c_3 \Lambda}{\sigma_1-\sigma_2}
\int_{P(\sigma_1)} f^{2}
 dyds\\
\leq &c_{12}\Lambda^2q^2\epsilon_3^{-1}\int_{P(\sigma_1)}f^2dyds
+\epsilon_3\int_{P(\sigma_1)}|\nabla(f\psi)|^2dyds + \frac{c_3
\Lambda}{\sigma_1-\sigma_2} \int_{P(\sigma_1)} f^{2}
 dyds.\eal\end{equation}

\vspace{1cm}
\textbf{$L^2-L^{\infty}$ Estimate:}
An $L^2-L^{\infty}$ bound is derived using Moser's iteration.
Recall inequality (\ref{T1T2T3T4}) from Step 1 and substitute the estimates
for $T_1\ (\ref{T1}),T_2\ (\ref{T2}) ,T_3\ (\ref{T3}),T_4\ (\ref{T4})$, found in Step 2-Step 5 to obtain:
\[
\al\int_{P(\sigma_1)}|&\nabla(f\psi)|^2dyds+\frac{1}{2}\int_{C(\sigma_1)}f^2(y,0)\phi^2(y)dy\\
&\leq\epsilon_1c_4K_b^{\frac{4}{3}}(C_{1,4,1})\int_{P(\sigma_1)}|\nabla(f\psi)|^2dyds+\frac{c_3\epsilon_1^{-3}}{(\sigma_1-\sigma_2)^4}\int_{P(\sigma_1)}f^2dyds\\
&\hspace{.5cm}+\frac{c_5}{(\sigma_1-\sigma_2)^2}\int_{P(\sigma_1)}f^2dyds\\
&\hspace{.5cm}
+\epsilon_2c_{11}\int_{P(\sigma_1)}|\nabla(f\psi)|^2dyds+\frac{c_6m^{\frac{m-2}{m-1}}\epsilon_2^{\frac{-1}{m-1}}}{(m-1)(\sigma_1-\sigma_2)^{\frac{m}{m-1}}}\int_{P(\sigma_1)}f^2dyds\\
&\hspace{.5cm}+\epsilon_3\int_{P(\sigma_1)}|\nabla(f\psi)|^2dyds+c_{12}\Lambda^2q^2\epsilon_3^{-1}\int_{P(\sigma_1)}f^2dyds.
\eal
\]
Choose \[\epsilon_1=\frac{1}{6c_4K_b^{\frac{4}{3}}(C_{1,4,1})},\
\epsilon_2=\frac{1}{6c_{11}},\ \epsilon_3=\frac{1}{6}\] and absorb
the appropriate terms to the left hand side. Then, we have the
following:
\[\al\int_{P(\sigma_1)}|&\nabla(f\psi)|^2dyds+\int_{C(\sigma_1)}f^2(y,0)\phi^2(y)dy\\&\leq
\left(\frac{c_{13}K_b^4(C_{1,4,1})}{(\sigma_1-\sigma_2)^4}+\frac{c_{14}}{(\sigma_1-\sigma_2)^2}+\frac{c_{15}m^{\frac{m-2}{m-1}}}{(m-1)(\sigma_1-\sigma_2)^{\frac{m}{m-1}}}+c_{16}q^2\Lambda^2\right)\int_{P(\sigma_1)}f^2dyds.\eal\]
Consequently,\begin{equation}\label{energyestimate}\al\int_{P(\sigma_1)}|\nabla(f\psi)|^2dyds+\int_{C(\sigma_1)}&f^2(y,0)\phi^2(y)dy\\
&\leq\frac{c_{17}q^2}{(\sigma_1-\sigma_2)^4}\left(K_b^4(C_{1,4,1})+\Lambda^2+1\right)\int_{P(\sigma_1)}f^2dyds.\eal\end{equation}
The last inequality follows with $q=1+\frac{2}{n}>1$ and
$0<\sigma_1-\sigma_2<1$, if $m$ is such that $\frac{m}{m-1}\leq4$.
This implies $m\geq\frac{4}{3},$ but our previous restriction on
$m$ required $1<m\leq\frac{4}{3}$. Thus, we let $m=\frac{4}{3}$
and deduce (\ref{energyestimate}) above.

\vspace{.5cm} \textbf{Moser's Iteration:} We claim that Moser's
iteration process and the estimate (\ref{energyestimate}) together
imply:
\[\sup_{P_{2,3,1}}\Omegabar_+^2\leq c_{21}\left(K_b^4(C_{1,4,1})+\Lambda^2+1\right)^{\frac{5}{2}}\int_{P_{1,4,1}}\Omegabar_+^2dyds.\] H\"{o}lder's inequality and the
Sobolev inequality imply:
\[\al\int_{\reals^n}(f\phi)^{2(1+\frac{2}{n})}dy&\leq\left(\int_{\reals^n}(f\phi)^2dy\right)^{\frac{2}{n}}\left(\int_{\reals^n}(f\phi)^{\frac{2n}{n-2}}dy\right)^{\frac{n-2}{n}}\\
&\leq
c_{18}\left(\int_{\reals^n}(f\phi)^2dy\right)^{\frac{2}{n}}\left(\int_{\reals^n}|\nabla
(f\phi)|^2dy\right).\eal\] Multiply by  the time portion of the
cut-off function to the correct power,
$\eta^{2(1+\frac{2}{n})}(s)$, on both sides and integrate over
time; one can deduce:
\[\int_{-\sigma_1^2}^0\int_{\reals^n}(f\psi)^{2(1+\frac{2}{n})}dyds\leq c_{18}\sup_{-\sigma_1^2\leq s\leq 0}\left(\int_{\reals^n}(f\psi)^2dy\right)^{\frac{2}{n}}\int_{-\sigma_1^2}^0\int_{\reals^n}|\nabla(f\psi)|^2dyds.\]
We use properties of the cut-off to obtain:
\begin{equation}\label{mosercutoff}\int_{P(\sigma_1)}(\psi
f)^{2(1+\frac{2}{n})}dyds\leq c_{18}\left(\sup_{-\sigma_1^2\leq
s<0}\int_{C(\sigma_1)}(f\psi)^2(y,s)dy\right)^{\frac{2}{n}}\int_{P(\sigma_1)}|\nabla(f\psi)|^2dyds.\end{equation}
In fact, with $n=3$ the above is:
\begin{equation}\label{mosern=3}\int_{P(\sigma_1)}(\psi f)^{\frac{10}{3}}dyds\leq c_{18}\left(\sup_{-\sigma_1^2\leq s<0}\int_{C(\sigma_1)}(f\phi)^2(y,s)dy\right)^{\frac{2}{3}}\int_{P(\sigma_1)}|\nabla (f\psi)|^2dyds.\end{equation}
We are noting this here because we will use this later in Section
4. The above argument can be run for each time level
$-\sigma_1^2\leq s< 0$ and in fact (\ref{energyestimate}) holds
for all $s$ in this interval as the upper time limit of the time
cut-off function. Thus, the second to last factor on the right hand side of
inequality (\ref{mosercutoff}) is still controlled by estimate
(\ref{energyestimate}).  So together with the estimate and the
cut-off function again, we get:
\begin{equation}\label{iteration1}\al \int_{P(\sigma_2)}\Omegabar_+^{2q\gamma}dyds&\leq
c_{18}\left(\frac{c_{16}q^2}{\tau^4}\left(K_b^4(C_{1,4,1})+\Lambda^2+1\right)\int_{P(\sigma_1)}\Omegabar_+^{2q}dyds\right)^{\gamma},\eal\end{equation}
where $\gamma=1+\frac{2}{n},\ \tau=\sigma_1-\sigma_2.$

Let $\tau_i=2^{-i-2},\  \sigma_0=1,\
\sigma_i=\sigma_{i-1}-\tau_i=1-\sum_{j=1}^i\tau_j,\ q=\gamma^i$.
Recall $P(\sigma_i)=P_{5-4\sigma_i,4\sigma_i,1}$. Then
(\ref{iteration1}) generalizes to:
\begin{equation}\label{mosergammai}\int_{P(\sigma_{i+1})}\Omegabar_+^{2\gamma^{i+1}}dyds\leq
c_{18}\left(c_{19}^{i+2}\gamma^{2i}\left(K_b^4(C_{1,4,1})+\Lambda^2+1\right)\int_{P(\sigma_i)}\Omegabar_+^{2\gamma^i}dyds\right)^{\gamma},\end{equation}

which, after taking the $\frac{1}{\gamma}$-th power of both sides, implies:
\[\left(\int_{P(\sigma_{i+1})}\Omegabar_+^{2\gamma^{i+1}}dyds\right)^{\frac{1}{\gamma}}\leq
c_{18}^{\frac{1}{\gamma}}\left(c_{19}^{i+2}\gamma^{2i}\left(K_b^4(C_{1,4,1})+\Lambda^2+1\right)\int_{P(\sigma_i)}\Omegabar_+^{2\gamma^i}dyds\right).\]

After iterating the above process, that is, using
(\ref{mosergammai}) on the integral on the left and raising both
sides to the $\frac{1}{\gamma}$-th power repeatedly, one obtains:
\[
\al
\left(\int_{P(\sigma_{i+1})}\right.&\left.\Omegabar_+^{2\gamma^{i+1}}dyds\right)^{\gamma^{-i-1}}\\
&\leq c_{18}^{\sum\gamma^{-j}}c_{19}^{\sum (j+1)\gamma
^{-j+1}}\gamma^{2\sum (j-1)\gamma^{-j+1}}\left(K_b^4(C_{1,4,1})+
\Lambda^2+1\right)^{\sum\gamma^{-j+1}}\int_{P_{1,4,1}}\Omegabar_+^2dyds.
\eal\] Note the sums in the exponents are all from $j=1$ to $j=i+1$.
Let $i\rightarrow\infty$. All the exponent series converge.  In
particular, the series in the exponent for
$\left(K_b^4(C_{1,4,1})+\Lambda^2+1\right)$ converges to
$\frac{5}{2}$. Note also that $\sigma_i\rightarrow\frac{3}{4}$, and
so:
\begin{equation}\label{estforOmega+}\sup_{P_{2,3,1}}\Omegabar_+^2\leq
c_{20}\left(K_b^4(C_{1,4,1})+\Lambda^2+1\right)^{\frac{5}{2}}\int_{P_{1,4,1}}\Omegabar_+^2dyds.\end{equation}

Next, repeating the argument on
$\Omegabar_-=\left\{\begin{array}{cc}-\Omega+\Lambda &
\Omega\leq0\\
\Lambda&\Omega>0\end{array}\right.$ yields:
\[\sup_{P_{2,3,1}}\Omegabar_-^2\leq
c_{20}\left(K_b^4(C_{1,4,1})+\Lambda^2+1\right)^{\frac{5}{2}}\int_{P_{1,4,1}}\Omegabar_-^2dyds.\]
Recall
\[\begin{array}{cc}\Omegabar_+=\left\{\begin{array}{cc}\Omega+\Lambda&\Omega\geq 0\\
\Lambda & \Omega<0\end{array}\right. & \Omegabar_-=\left\{\begin{array}{cc}-\Omega+\Lambda&\Omega\leq 0\\
\Lambda & \Omega>0\end{array}\right.\\
\ &\ \\\Omega=\Omegabar_+-\Omegabar_- &
\Lambda=\|v_{\theta}\|_{L^{\infty}(P_{1,4,1})}\leq
\|rv_{0,\theta}\|_{L^{\infty}(\reals^3)}\end{array}\] Thus,
\[\al\sup_{P_{2,3,1}}\Omega^2&\leq \sup_{P_{2,3,1}}\left(\Omegabar_+-\Omegabar_-\right)^2\\
&\leq
c_{20}\left(K_b^4(C_{1,4,1})+\Lambda^2+1\right)^{\frac{5}{2}}\sup_{P_{2,3,1}}\left(\Omegabar_+^2+\Omegabar_-^2\right)\\
&\leq
c_{20}\left(K_b^4(C_{1,4,1})+\Lambda^2+1\right)^{\frac{5}{2}}\left(\int_{P_{1,4,1}}\Omegabar_+^2dyds+\int_{P_{1,4,1}}\Omegabar_-^2dyds\right)\\
&\leq
c_{20}\left(K_b^4(C_{1,4,1})+\Lambda^2+1\right)^{\frac{5}{2}}\left(\int_{\{\Omega\geq0\}}(\Omega+\Lambda)^2dyds+\int_{\{\Omega<0\}}\Lambda^2dyds\right.\\
&\hspace{6cm}\left.+\int_{\{\Omega\leq0\}}(-\Omega+\Lambda)^2dyds+\int_{\{\Omega>0\}}\Lambda^2dyds\right)\\
&\leq
c_{20}\left(K_b^4(C_{1,4,1})+\Lambda^2+1\right)^{\frac{5}{2}}\left(\int_{P_{1,4,1}}(\Omega+\Lambda)^2+(-\Omega+\Lambda)^2+2\Lambda^2dyds\right)\\
&=c_{20}\left(K_b^4(C_{1,4,1})+\Lambda^2+1\right)^{\frac{5}{2}}\left(2\int_{P_{1,4,1}}\Omega^2dyds+4\int_{P_{1,4,1}}\Lambda^2dyds\right)\\
&\leq
c_{21}\left(K_b^4(C_{1,4,1})+\Lambda^2+1\right)^{\frac{5}{2}}\left(\|\Omega\|_{L^2(P_{1,4,1})}^2+\Lambda^2\right).\eal\]

\vspace{.5cm}
\textbf{Re-scaling:} We now recall
that we omitted the "tildes" in the notation in the above
computations.  So what has actually been proven thus far is:
\[\sup_{(\~{x},\~{t})\in P_{2,3,1}}\~{\Omega}^2(\~{x},\~{t})\leq c_{21}\left(K^4_{\~{b}}(C_{1,4,1})+\~{\Lambda}^2+1\right)^{\frac{5}{2}}\left(\|\~{\Omega}\|_{L^2(P_{1,4,1})}^2+\~{\Lambda}^2\right).\]

 Recall $\~{x}=\frac{x}{k},\
\~{t}=\frac{t}{k^2},\
\~{\Omega}(\~{x},\~{t})=\frac{\widetilde{\omega_{\theta}}(\~{x},\~{t})}{\~{r}}$.
So with $2\leq\~{r}\leq 3$ on the left and $1\leq\~{r}\leq 4$ on
the right we can derive:
\[\sup_{(\~{x},\~{t})\in P_{2,3,1}}\~{\omega}^2_{\theta}(\~{x},\~{t})\leq c_{22}\left(K_{\~{b}}^4(C_{1,4,1})+\~{\Lambda}^2+1\right)^{\frac{5}{2}}\left(\int_{P_{1,4,1}}\~{\omega}^2_{\theta}(\~{x},\~{t})d\~{x}d\~{t}+\~{\Lambda}^2\right).\]
We recall from the Section 2 Preliminaries :
\[
K_{\~{b}}(C_{1,4,1})=\|\~{b}(\~{x},\~{t})\|_{L^{\infty}(-1,0;L^2(C_{1,4,1}))}=\frac{1}{k^{\frac{1}{2}}}\|b(x,t)\|_{L^{\infty}{(-k^2,0;L^2(C_{1,4,k}))}}\]
and
\[\|\~{\omega}(\~{x},\~{t})\|_{L^2(P_{1,4,1})}=\frac{1}{k^{\frac{1}{2}}}\|\omega(x,t)\|_{L^2(P_{1,4,k})}.\]
Also we note the control on $\Lambda$ is a scaling invariant
quantity. Since $\Lambda=\|v_{\theta}\|_{L^{\infty}(P_{1,4,1})}$, we
use Proposition \ref{chaeandleeprop}:
\[\al \~{\Lambda}&=\left(\sup_{P_{1,4,1}}|\~{v_{\theta}}(\~{x},\~{t})|\right)\\
&\leq \left(||\~{r}\~{v_{\theta}}(\~{x},-T)\|_{L^{\infty}(\reals^3)}\right) \hspace{1cm}\textrm{applying Proposition \ref{chaeandleeprop},}\\
&=\|rv_{\theta}(x,-T)\|_{L^{\infty}(\reals^3)}\\
&=\|rv_{0,\theta}\|_{L^{\infty}(\reals^3)}.\eal\]

We utilize $0<k<1$ to obtain:
\[\al&\sup_{(x,t)\in P_{2,3,k}}k^4\omega ^2_{\theta}(x,t)\\&\leq
c_{22}\left(\frac{1}{k^2}\|b\|_{L^{\infty}(-k^2,0;L^2(C_{1,4,k}))}^4+\|rv_{0,\theta}\|_{L^{\infty}(\reals^3)}^2\right)^{\frac{5}{2}}\left(\int_{P_{1,4,k}}k^4\omega_{\theta}^2(x,t)\frac{1}{k^5}dxdt+\|rv_{0,\theta}\|_{L^{\infty}(\reals^3)}^2\right)\\
&\leq
\frac{c_{23}}{k^6}\left(\|b\|_{L^{\infty}(-k^2,0;L^2({C_{1,4,k}}))}^2+k\|rv_{0,\theta}\|_{L^{\infty}(\reals^3)}\right)^5\left(\|\omega_{\theta}\|_{L^2(P_{1,4,k})}^2+k\|rv_{0,\theta}\|_{L^{\infty}(\reals^3)}^2\right).\\
\eal\] Therefore,
\[\al
\|\omega_{\theta}(x,t)&\|_{L^{\infty}(P_{2,3,k})}\\
&\leq\frac{c_{24}}{k^{5}}\left(\|b\|_{L^{\infty}(-k^2,0;L^2(C_{1,4,k}))}^2+k\|rv_{0,\theta}\|_{L^{\infty}(\reals^3)}\right)^{\frac{5}{2}}\left(\|\omega_{\theta}\|_{L^2(P_{1,4,k})}+\sqrt{k}\|rv_{0,\theta}\|_{L^{\infty}(\reals^3)}\right).\eal\]
This proves part (i) of Theorem \ref{ourtheorem}.

Note, the way the cubes on the left and right are related is that
on the right, we have $\frac{1}{2}$ of the inner radius and
$\frac{4}{3}$
of the outer radius. %So in the next section, since we want the
%point-wise bound over $P_{1,4,1}$, we will have $\omega_{\theta}$
%bounded by it's $L^2$ norm on the cube
%$P_{\frac{1}{2},\frac{16}{3},1}$.
\end{section}

\begin{section}{A priori bounds for $\omega_r$ and $\omega_z$}
In this section we use the a priori bound established in part (i)
of Theorem 1.1 (ie. $|\omega_{\theta}|\leq\frac{B_1}{r^5}$) and
the $2\times 2$ system below, which consists of the two remaining
curl equations noted before, to derive a priori bounds for
$\omega_r$ and $\omega_z$.
\begin{equation}\label{2x2system}\left\{\al&\Delta\omega_r-(b\cdot\nabla)\omega_r+\omega_r\left(\frac{\partial v_r}{\partial
r}-\frac{1}{r^2}\right)+\omega_z\frac{\partial v_r}{\partial
z}-\frac{\partial
\omega_r}{\partial t}=0,\\
&\Delta\omega_{z}-(b\cdot\nabla)\omega_{z}+\omega_z\frac{\partial
v_z}{\partial z}+\omega_{r}\frac{\partial v_z}{\partial
r}-\frac{\partial \omega_{z}}{\partial
t}=0.\eal\right.\end{equation} The drift term, $b\cdot\nabla$ can
be dealt with in a similar manner to that in Section 3.  The main
work is to treat the potential terms where
$\displaystyle\frac{\partial v_r}{\partial r}-\frac{1}{r^2},\
\frac{\partial v_r}{\partial z},\ \frac{\partial v_z}{\partial
r},\ \frac{\partial v_z}{\partial r}$ are regarded as potentials.
It turns out one can control the $L^{\frac{10}{3}}$ norm of these
using the a priori bound on $\omega_{\theta}$ established in part
(i) of Theorem \ref{ourtheorem} and the a priori bound on
$v_{\theta}$ from Proposition \ref{chaeandleeprop}. These
$L^{\frac{10}{3}}$ bounds are sufficient to prove part (ii) of
Theorem \ref{ourtheorem}.

We need two lemmas which are localized versions of Lemma 2 and
Lemma 3 in \cite{NP}, and very similar, also, to Lemma 3 in
\cite{CL}. Both should be known, but the proofs are short and are
included here for completeness.  First we recall our notation,
$C_{A,B,R}=\{(r,\theta,z)|\ AR\leq r\leq BR,\ 0\leq\theta\leq
2\pi,\ |z|\leq BR\}\subset\reals^3$, and
$P_{A,B,R}=C_{A,B,R}\times(-R^2,0)$.

\vspace{.5cm}\begin{lemma}\label{lemma1}Let $v\in
C^{\infty}(C_{1,4,1})$ be a vector field. Then, for all $q>1$, there
exists a constant, $c(q)>0$, such that\[\|\nabla v\|_{L^q
(C_{2,3,1})}\leq c(q)\left(\|\Curl v\|_{L^q(C_{1,4,1})}+\|\Div
v\|_{L^q(C_{1,4,1})}+\|v\|_{L^q(C_{1,4,1})}\right).\]
\end{lemma}
\begin{proof}Define $\phi$ to be a cut-off function such that $\phi\in C_0^{\infty}(C_{1,4,1}),\ 0\leq\phi\leq1,\ \phi=1$ in $C_{2,3,1},\ |\nabla\phi|\leq c_1$, a constant. Then $v\phi$ is compactly supported, and it is well known that:
\begin{equation}\label{nplemma}\|\nabla (v\phi)\|_{L^q(C_{1,4,1})}\leq c(q)\left(\|\Curl (v\phi)\|_{L^q(C_{1,4,1})}+\|\Div (v\phi)\|_{L^q(C_{1,4,1})}\right).\end{equation}
(This is sometimes called the Helmholtz or Hodge decomposition). Next note\[\al\Div(v\phi)&=\Div v\ \phi+v\cdot\nabla\phi\hspace{1cm}\textrm{ and }\\
\Curl(v\phi)&=\Curl v\ \phi+\nabla\phi\times v.\eal\] The lemma
follows by substituting the last two identities into the right
hand side of (\ref{nplemma}) and using the Minkowski inequality
and properties of the cutoff function.
\end{proof}

 The following lemma is a generalization of Lemma 3 in \cite{NP}.
\begin{lemma}\label{lemma2}Let $v=v(x,t)$ be a
divergence free, axis symmetric, smooth vector field in
$Q_{1,4}=C_{1,4,1}\times[-T,T]$ for fixed $T>0$.  Then, for all
$q>1$, there exists a constant, $c=c(q)>0$, such that
\[\al\|\nabla
v_r\|_{L^q(Q_{2,3})}+\left\|\frac{v_r}{r}\right\|_{L^q(Q_{2,3})}+\|\nabla
v_z&\|_{L^q(Q_{2,3})}\\&\leq c(q)\left(\|(\Curl
v)_{\theta}\|_{L^q(Q_{1,4})}+\|v\|_{L^q(Q_{1,4})}\right).\eal\]
\end{lemma}
\begin{proof}In the cylindrical coordinate system, for an axis symmetric vector field, $\Div v=0$ means
\[\frac{\partial v_r}{\partial r}+\frac{v_r}{r}+\frac{\partial v_z}{\partial z}=0.\]
Therefore the vector field:
\[\bar{v}=v_r\->{e_r}+v_z\->{e_z}\]
is still divergence free.
Since the inequality we want to prove does not involve $v_{\theta}$, we first work on $\bar{v}$
where $v_{\theta}$ is not involved. Also $\bar{v}$ is axis symmetric, and so $\Curl \bar{v}$
has only one nonzero component, the one in the direction of $\overrightarrow{e_{\theta}}$.
This is because for axis symmetric vector fields:
\[\al\omega(x,t)=\omega_r\->{e_r}+\omega_{\theta}&\->{e_{\theta}}+\omega_z\->{e_z}\\
\ &\ \\
\omega_r=-\frac{\partial v_{\theta}}{\partial z},\ \omega_{\theta}=\frac{\partial v_r}{\partial z}-\frac{\partial v_z}{\partial r}&,\ \omega_z=\frac{\partial v_{\theta}}{\partial r}+\frac{v_{\theta}}{r}.\eal\]
Thus,
\[\Curl \bar{v}=(\Curl \bar{v})_{\theta}\->{e_{\theta}}.\]
Applying Lemma \ref{lemma1} on $\bar{v}$, we deduce, for any fixed $t$:
\[\al\|\nabla\bar{v}(\cdot,t)\|_{L^q(C_{2,3,1})}&\leq c(q)\left(\|\Curl \bar{v}(\cdot,t)\|_{L^q(C_{1,4,1})}+\|\bar{v}(\cdot,t)\|_{L^q(C_{1,4,1})}\right)\\
&=c(q)\left(\|(\Curl\bar{v})_{\theta}(\cdot,t)\|_{L^q(C_{1,4,1})}+\|\bar{v}(\cdot,t)\|_{L^q(C_{1,4,1})}\right).\eal\]
Note $\displaystyle(\Curl v)_{\theta}=\frac{\partial v_r}{\partial z}-\frac{\partial v_z}{\partial r}=(\Curl \bar{v})_{\theta}$, and so:
\[\|\nabla\bar{v}(\cdot,t)\|_{L^q(C_{2,3,1})}\leq c(q)\left(\|(\Curl v)_{\theta}(\cdot,t)\|_{L^q(C_{1,4,1})}+\|v(\cdot,t)\|_{L^q(C_{1,4,1})}\right).\]
Thus,\begin{equation}\label{afterbar}\al\|\nabla
v_r(\cdot,t)\|_{L^q(C_{2,3,1})}+\|\nabla
&v_z(\cdot,t)\|_{L^q(C_{2,3,1})}+\left\|\frac{v_r(\cdot,t)}{r}\right\|_{L^q(C_{2,3,1})}\\&\leq
c(q)\left(\|(\Curl
v)_{\theta}(\cdot,t)\|_{L^q(C_{1,4,1})}+\|v(\cdot,t)\|_{L^q(C_{1,4,1})}\right).\eal\end{equation}
Here,
$\displaystyle\left\|\frac{v_r(\cdot,t)}{r}\right\|_{L^q(C_{2,3,1})}$
is bounded due to the inequality:
\[\left\|\frac{v_r(\cdot,t)}{r}\right\|_{L^q(C_{2,3,1})}\leq \left\|\frac{\partial v_r(\cdot,t)}{\partial r}\right\|_{L^q(C_{2,3,1})}+\left\|\frac{\partial v_z(\cdot,t)}{\partial z}\right\|_{L^q(C_{2,3,1})},\]
which comes from the divergence free equation.  Taking the q-th
power on (\ref{afterbar}) and integrating in time, we deduce the
lemma.
\end{proof}

Taking $q=\frac{10}{3}$ in Lemma \ref{lemma2} yields the following
Proposition:
\begin{proposition}\label{lemmawith10/3}
For $v$, a smooth, axis symmetric solution to the Navier-Stokes
equations in $Q_{1,4}$, then there exists a constant $c_1>0$ such
that:
\[ \al\|\nabla
v_r\|_{L^{\frac{10}{3}}(Q_{2,3})}+\left\|\frac{v_r}{r}\right\|_{L^{\frac{10}{3}}(Q_{2,3})}+\|\nabla
v_z&\|_{L^{\frac{10}{3}}(Q_{2,3})}\\&\leq
c_1\left(\|\omega_{\theta}\|_{L^{\frac{10}{3}}(Q_{1,4})}+\|v\|_{L^{\frac{10}{3}}(Q_{1,4})}\right).\eal\]\end{proposition}
The right hand side is a priori bounded due to standard energy
estimates and our Theorem \ref{ourtheorem} (i).

\vspace{.5cm} \textbf{Proof of Theorem \ref{ourtheorem} (ii)}:

 We use the scaling invariance of (\ref{2x2system}) and do the
analysis in $P_{1,4,1}\subset Q_{1,4}$. We let $V$ be the matrix:
\[V=\left[\begin{array}{cc}
\frac{\partial v_r}{\partial r}-\frac{1}{r^2} & \frac{\partial v_z}{\partial r}\\
\ &\ \\
\frac{\partial v_r}{\partial z} & \frac{\partial v_z}{\partial z}
\end{array}\right],\]which can be regarded as a potential in the system when we take the equations together. Proposition \ref{lemmawith10/3} shows $V\in
L^{\frac{10}{3}}(P_{1,4,1}).$ This, along with our analysis on the
drift term $b$ as before implies, by a similar argument to that in
Section 3, that $\omega_r$ and $\omega_z$ are also a priori
bounded. Again, scaling, and in particular the scaling of
$\|V\|_{L^{\frac{10}{3}}(P_{1,4,1})}$, will come into play.

We let $q\geq1$ be a rational number and choose
$\psi=\phi(y)\eta(s)$ to be the same refined cut-off function as
previously defined, satisfying the following:\[\supp\phi \subset
C(\sigma_1);\ \phi(y)=1 \textrm{ for all }\in C(\sigma_2);\
\frac{|\nabla \phi|}{\phi
^{\delta}}\leq\frac{c_2}{\sigma_1-\sigma_2}\textrm{ for }\delta
\in (0,1),\ 0\leq \phi\leq 1;\
\] \[\supp \eta \subset (-\sigma_1^2,0];\
\eta(s)=1 \textrm{ for all } s\in [-\sigma_2^2,0];\ |\eta '|\leq
\frac{c_3}{(\sigma_1-\sigma_2)^2};\ 0\leq \eta \leq 1.\] We start
by using $\omega_r^{2q-1}\psi^2$ as a test function on the first
equation of system (\ref{2x2system}).
\[\al0=&\int_{P(\sigma_1)}\left(\Delta \omega_r-b\cdot\nabla\omega_r+\omega_r\left(\frac{\partial v_r}{\partial r}-\frac{1}{r^2}\right)+\omega_z\frac{\partial v_r}{\partial_z}-\frac{\partial\omega_r}{\partial s}\right)\omega_r^{2q-1}\psi^2dyds\\
=&\int_{P(\sigma_1)}\omega_r^{2q-1}\psi^2\Delta\omega_rdyds\\
&\hspace{1cm}-\int_{P(\sigma_1)}\frac{1}{q}b\cdot\nabla(\omega_r^q)(\omega_r^q\psi^2)dyds-\int_{P(\sigma_1)}\frac{1}{q}\partial_s(\omega_r^q)(\omega_r^q\psi^2)dyds\\
&\hspace{2.5cm}+\int_{P(\sigma_1)}\left(\frac{\partial
v_r}{\partial
r}-\frac{1}{r^2}\right)(\omega_r^{2q}\psi^2)+\left(\frac{\partial
v_r}{\partial z}\right)\omega_z\omega_r^{2q-1}\psi^2dyds.\eal\]
We work on the first term on the right hand side, using integration
by parts, as usual, direct calculations, and algebraic
manipulations:
\[\al\int_{P(\sigma_1)}\omega_r^{2q-1}\psi^2\Delta\omega_rdyds&=-\int_{P(\sigma_1)}\nabla(\omega_r^{2q-1}\psi^2)\cdot\nabla\omega_rdyds\\
&=-\int_{P(\sigma_1)}(2q-1)(\omega_r^{2q-2}\nabla\omega_r)\cdot\nabla\omega_r\psi^2+\omega^{2q-1}\nabla\omega_r\cdot\nabla(\psi^2)dyds\\
&=-\int_{P(\sigma_1)}(2q-1)(\omega_r^{q-1}\nabla\omega_r)\cdot(\omega_r^{q-1}\nabla\omega_r)\psi^2+\nabla(\psi^2)\omega_r^q(\omega_r^{q-1}\nabla\omega_r)dyds\\
&=-\frac{2q-1}{q^2}\int_{P(\sigma_1)}\nabla(\omega_r^q)\cdot\nabla(\omega_r^q)\psi^2dyds-\frac{1}{q}\int_{P(\sigma_1)}\omega_r^q\nabla(\omega_r^q)\cdot\nabla(\psi^2)dyds\\
&\leq -\frac{1}{q}\int_{P(\sigma_1)}\nabla(\omega_r^q)\cdot\left(\nabla(\omega_r^q)\psi^2+\nabla(\psi^2)\omega_r^q\right)dyds,\hspace{1.cm}\textrm{ since $\frac{1}{q}<\frac{2q-1}{q^2}$,}\\
&=-\frac{1}{q}\int_{P(\sigma_1)}\nabla(\omega_r^q)\cdot\nabla(\omega_r^q\psi^2)dyds\\
&=-\frac{1}{q}\int_{P(\sigma_1)}\left(|\nabla(\omega_r^q\psi)|^2-|\nabla\psi|^2\omega_r^{2q}\right)dyds.\eal\]
This implies:
\begin{equation}\label{testforomegar}\al\int_{P(\sigma_1)}&|\nabla(\omega_r^q\psi)|^2dyds\\&\leq-\int_{P(\sigma_1)}b\cdot\nabla(\omega_r^q)(\omega_r^q\psi^2)dyds-\int_{P(\sigma_1)}\partial_s(\omega_r^q)(\omega_r^q\psi^2)dyds+\int_{P(\sigma_1)}|\nabla\psi|^2\omega_r^{2q}dyds\\
&\hspace{1cm}+q\int_{P(\sigma_1)}\left[\left(\frac{\partial
v_r}{\partial
r}-\frac{1}{r^2}\right)(\omega_r^{2q}\psi^2)+\left(\frac{\partial
v_r}{\partial
z}\right)\omega_z\omega_r^{2q-1}\psi^2\right]dyds.\eal\end{equation}

Similarly, we use $\omega_z^{2q-1}\psi^2$ as a test function in the second equation in system (\ref{2x2system}) to arrive at:
\begin{equation}\label{testforomegaz}\al\int_{P(\sigma_1)}&|\nabla(\omega_z^q\psi)|^2dyds\\&\leq-\int_{P(\sigma_1)}b\cdot\nabla(\omega_z^q)(\omega_z^q\psi^2)dyds-\int_{P(\sigma_1)}\partial_s(\omega_z^q)(\omega_z^q\psi^2)dyds+\int_{P(\sigma_1)}|\nabla\psi|^2\omega_z^{2q}dyds\\
&\hspace{1cm}+q\int_{P(\sigma_1)}\left[\left(\frac{\partial
v_z}{\partial z}\right)(\omega_z^{2q}\psi^2)+\left(\frac{\partial
v_z}{\partial
r}\right)\omega_r\omega_z^{2q-1}\psi^2\right]dyds.\eal\end{equation}

We let $f=|\omega_r|^q+|\omega_z|^q$ and $V$ represent the matrix:
\[V=\left[\begin{array}{cc}
\frac{\partial v_r}{\partial r}-\frac{1}{r^2} & \frac{\partial v_z}{\partial r}\\
\ &\ \\
\frac{\partial v_r}{\partial z} & \frac{\partial v_z}{\partial z}
\end{array}\right].\]

We add (\ref{testforomegar}) and (\ref{testforomegaz}) and apply
Cauchy-Schwartz inequality to the term involving $V$ to obtain:
\[\int_{P(\sigma_1)}|\nabla (f\psi)|^2dyds\leq2\int_{P(\sigma_1)}\left(-b\cdot\nabla f(f\psi^2)-\partial_s f(f\psi^2)+|\nabla\psi|^2f^2+qc_5|V|f^2\psi^2\right)dyds.\]
Here $|V|$ is the max norm of the matrix.  We proceed just as in
the end of Step 1 in Section 3 to reach:
\begin{equation}\label{T1T2T3}\al\int_{P(\sigma_1)}|\nabla(f\psi)|^2&dyds+\frac{1}{2}\int_{C(\sigma_1)}f^2(y,0)\phi^2(y)dy\\
&\leq-\int_{P(\sigma_1)}2b\cdot\nabla f(f\psi^2)dyds+2\int_{P(\sigma_1)}(\eta\partial_s\eta+|\nabla\psi|^2)f^2dyds\\
&\hspace{3cm}+c_4q\int_{P(\sigma_1)}|V|f^2\psi^2dyds\\&:=T_1+T_2+T_3.\eal\end{equation}

Terms $T_1$ and $T_2$ are in the same form as  to $T_1$ and $T_2$
in (\ref{T1T2T3T4}) of Section 3. Therefore, they are treated in
an identical manner as found there. We recall the estimates on
those terms now( see (\ref{T1}) and (\ref{T2})):
\begin{equation}\label{T1'}|T_1|\leq
\epsilon_1c_5K_b^{\frac{4}{3}}(C_{1,4,1})\int_{P(\sigma_1)}|\nabla(f\psi)|^2dyds+\frac{c_6\epsilon_1^{-3}}{(\sigma_1-\sigma_2)^4}\int_{P(\sigma_1)}f^2dyds\end{equation}
\begin{equation}\label{T2'}
|T_2|\leq
\frac{c_7}{(\sigma_1-\sigma_2)^2}\int_{P(\sigma_1)}f^2dyds.\end{equation}
We proceed to term $T_3$ involving the matrix constructed from the
potential terms in system (\ref{2x2system}). We employ
H\"{o}lder's inequality twice here:
\[\al T_3&=c_4q\int_{P(\sigma_1)}|V|(f\psi)^2dyds\\
&\leq c_4q\left(\int_{P(\sigma_1)}|V|^{\frac{10}{3}}dyds\right)^{\frac{3}{10}}\left(\int_{P(\sigma_1)}\left((f\psi)^2\right)^{\frac{10}{7}}dyds\right)^{\frac{7}{10}}\\
&=c_4q\|V\|_{L^{\frac{10}{3}}(P(\sigma_1))}\left(\int_{P(\sigma_1)}(f\psi)^{\frac{20}{7}}dyds\right)^{\frac{7}{10}}\\
&\leq c_4q\|V\|_{L^{\frac{10}{3}}(P_{1,4,1})}\left(\int_{P(\sigma_1)}(f\psi)^{\frac{20}{7}-a}(f\psi)^adyds\right)^{\frac{7}{10}}\hspace{2cm}0<a<\frac{20}{7}\\
&=c_4q\|V\|_{L^{\frac{10}{3}}(P_{1,4,1})}\left(\int_{P(\sigma_1)}(f\psi)^{(\frac{20}{7}-a)p}dyds\right)^{\frac{7}{10p}}\left(\int_{P(\sigma_1)}(f\psi)^{ap'}dyds\right)^{\frac{7}{10p'}}\eal\]
for $1<p,p'<\infty,\ \frac{1}{p}+\frac{1}{p'}=1$. If
$\left(\frac{20}{7}-a\right)p=\frac{10}{3}$ and $ap'=2$, then
$p=\frac{14}{9}$ and $p'=\frac{14}{5}$ and we get:
\[\al T_3&\leq c_4q\|V\|_{L^{\frac{10}{3}}(P_{1,4,1})}\left(\int_{P(\sigma_1)}(f\psi)^\frac{10}{3}dyds\right)^{\frac{9}{20}}\left(\int_{P(\sigma_1)}(f\psi)^2dyds\right)^{\frac{1}{4}}.\eal\]
We apply Young's inequality with exponents $\frac{4}{3}$ and $4$:
\begin{equation}\label{T3'}\al T_3&\leq \left[\left(\frac{4\epsilon_2}{3}\right)^{\frac{3}{4}}\left(\int_{P(\sigma_1)}(f\psi)^\frac{10}{3}dyds\right)^{\frac{9}{20}}\right]\times\\
&\hspace{5cm}\left[c_4q\left(\frac{4\epsilon_2}{3}\right)^{-\frac{3}{4}}\|V\|_{L^{\frac{10}{3}}(P_{1,4,1})}\left(\int_{P(\sigma_1)}(f\psi)^2dyds\right)^{\frac{1}{4}}\right]\\
&\leq\epsilon_2\left(\int_{P(\sigma_1)}((f\psi)^2)^{\frac{5}{3}}dyds\right)^{\frac{3}{5}}+c_8q^4\epsilon_2^{-3}\|V\|_{L^{\frac{10}{3}}(P_{1,4,1})}^4\int_{P(\sigma_1)}(f\psi)^2dyds\\
&\leq\epsilon_2\|(f\psi)^2\|_{L^{\frac{5}{3}}(P(\sigma_1))}+c_8q^4\epsilon_2^{-3}\|V\|_{L^{\frac{10}{3}}(P_{1,4,1})}^4\int_{P(\sigma_1)}f^2dyds.\eal\end{equation}
Note, $\|V\|_{L^{\frac{10}{3}}(P_{1,4,1})}$ can be controlled as a
result of Proposition \ref{lemmawith10/3}.

At this time we utilize in (\ref{T1T2T3}) the estimates for $T_1$ (\ref{T1'}), $T_2$
(\ref{T2'}), and $T_3$ (\ref{T3'}), which then
becomes:
\[\al&\int_{P(\sigma_1)}|\nabla(f\psi)|^2dyds+\frac{1}{2}\int_{C(\sigma_1)}f^2(y,0)\phi^2(y)dy\\
&\leq\epsilon_1c_5K_b^{\frac{4}{3}}(C_{1,4,1})\int_{P(\sigma_1)}|\nabla(f\psi)|^2dyds+\frac{c_6\epsilon_1^{-3}}{(\sigma_1-\sigma_2)^4}\int_{P(\sigma_1)}f^2dyds+\frac{c_7}{(\sigma_1-\sigma_2)^2}\int_{P(\sigma_1)}f^2dyds\\
&\hspace{1cm}+\epsilon_2\|(f\psi)^2\|_{L^{\frac{5}{3}}(P(\sigma_1))}+c_8q^4\epsilon_2^{-3}\|V\|_{L^{\frac{10}{3}}(P_{1,4,1})}^4\int_{P(\sigma_1)}f^2dyds.\eal\]
Choose\[\epsilon_1=\frac{1}{2c_5K_b^{\frac{4}{3}}(C_{1,4,1})}\] and
absorb the appropriate term the left hand side.  We arrive at:
\begin{equation}\label{energyestimate'}\al\int_{P(\sigma_1)}&|\nabla(f\psi)|^2dyds+\int_{C(\sigma_1)}f^2(y,0)\phi^2(y)dy\\
&\leq\frac{c_{9}K_b^4(C_{1,4,1})}{(\sigma_1-\sigma_2)^4}\int_{P(\sigma_1)}f^2dyds+\frac{c_{10}}{(\sigma_1-\sigma_2)^2}\int_{P(\sigma_1)}f^2dyds\\
&\hspace{1cm}+2\epsilon_2\|(f\psi)^2\|_{L^{\frac{5}{3}}(P(\sigma_1))}+c_{11}q^4\|V\|_{L^{\frac{10}{3}}(P_{1,4,1})}^4\int_{P(\sigma_1)}f^2dyds\\
&\leq2\epsilon_2\|(f\psi)^2\|_{L^{\frac{5}{3}}(P(\sigma_1))}+\frac{c_{12}q^4}{(\sigma_1-\sigma_2)^4}\left(K_b^4(C_{1,4,1})+\|V\|_{L^{\frac{10}{3}}(P_{1,4,1})}^4+1\right)\int_{P(\sigma_1)}f^2dyds,\eal\end{equation}
noting $0<\sigma_1-\sigma_2<1$ and $q=1+\frac{2}{n}>1$.

 Now,
recall (\ref{mosern=3}) in Moser's iteration in Section 3, which
follows from H\"{o}lder's inequality, the Sobolev inequality,
$n=3$, and properties of the cut-off function.  We have:
\[\int_{P(\sigma_1)}(\psi f)^{\frac{10}{3}}dyds\leq c_{13}\left(\sup_{-1\leq s<0}\int_{C(\sigma_1)}(f(y,s)\phi(y))^2dy\right)^{\frac{2}{3}}\int_{P(\sigma_1)}|\nabla(f\psi)|^2dyds.\]
Apply estimate (\ref{energyestimate'}), as we did in Section 3,
and take the $\frac{3}{5}$ power of both sides:
\[\al\|(\psi f)^2\|_{L^{\frac{5}{3}}(P(\sigma_1))}&\leq
\frac{c_{14}q^4}{(\sigma_1-\sigma_2)^4}\left(K_b^4(C_{1,4,1})+\|V\|_{L^{\frac{10}{3}}(P_{1,4,1})}^4+1\right)\int_{P(\sigma_1)}f^2dyds\\
&\hspace{2cm}+2c_{15}\epsilon_2\|(f\psi)^2\|_{L^{\frac{5}{3}}(P(\sigma_1))}.\eal\]
Choose
\[\epsilon_2=\frac{1}{4c_{15}},\] absorb the appropriate term to the left, take the $\frac{5}{3}$
power of both sides, use the cut-off function, and recall
$f=|\omega_r|^q+|\omega_z|^q$.  We get:
\[\int_{P(\sigma_2)}(|\omega_r|^q+|\omega_z|^q)^{2\gamma}\leq
c_{16}\left[\frac{c_{17}q^4}{\tau^4}\left(K_b^4(C_{1,4,1})+\|V\|_{L^{\frac{10}{3}}(P_{1,4,1})}^4+1\right)\int_{P(\sigma_1)}(|\omega_r|^q+|\omega_z|^q)^2dyds\right]^{\gamma},\]
where $\gamma=1+\frac{2}{n},\ n=3,\ \tau=\sigma_1-\sigma_2.$
Define $h(x,t)=\max(|\omega_r|,|\omega_z|)$ and observe $h^q\leq |\omega_r|^q+|\omega_z|^q\leq2h^q$. And so:
\begin{equation}\label{mosergammair}\int_{P(\sigma_2)}h^{2q\gamma}dyds \leq
c_{16}\left[\frac{c_{18}q^4}{\tau^4}\left(K_b^4(C_{1,4,1})+\|V\|_{L^{\frac{10}{3}}(P_{1,4,1})}^4+1\right)\int_{P(\sigma_1)}h^{2q}dyds\right]^{\gamma}.\end{equation}  Let $\tau_i=2^{-i-2},\ \sigma_0=1,\ \sigma_i=\sigma_{i-1}-\tau_i=1-\sum_{j=1}^i\tau_j,\ q=\gamma^i$. Thus we have
an analogue to (3.6):
\begin{equation}\label{mosergammair}\int_{P(\sigma_{i+1})}h^{2\gamma^{i+1}}dyds \leq
 c_{16}\left[c_{19}^{i+2}\gamma^{4i}\left(K_b^4(C_{1,4,1})+\|V\|_{L^{\frac{10}{3}}(P_{1,4,1})}^4+1\right)\int_{P(\sigma_1)}h^{2\gamma^i}dyds\right]^{\gamma}.\end{equation}
Raising both sides to the $\frac{1}{\gamma}$-th power, we get:
\[\left(\int_{P(\sigma_{i+1})}h^{2\gamma^{i+1}}dyds\right)^{\frac{1}{\gamma}}\leq c_{16}^{\frac{1}{\gamma}}\left[c_{19}^{i+2}\gamma^{4i}\left(K_b^4(C_{1,4,1})+\|V\|_{L^{\frac{10}{3}}(P_{1,4,1})}^4+1\right)\int_{P(\sigma_1)}h^{2\gamma^i}dyds\right].\]
Now we apply (\ref{mosergammair}) to the
integral on the right hand side, with $i$ replaced with $i-1$, to
obtain:
\[\al\left(\int_{P(\sigma_{i+1})}h^{2\gamma^{i+1}}dyds\right)^{\frac{1}{\gamma}}
 &\leq c_{16}^{\frac{1}{\gamma}}\left[c_{19}^{i+2}\gamma^{4i}\left(K_b^4(C{1,4})+\|V\|_{L^{\frac{10}{3}}(P_{1,4,1})}^4+1\right)\right]\times\\
 &\hspace{1cm}2c_{16}\left[c_{19}^{i+2}\gamma^{4i}
 \left(K_b^4(C_{1,4,1})+\|V\|_{L^{\frac{10}{3}}(P_{1,4,1})}^4+1\right)
 \int_{P(\sigma_{i-1})}h^{2\gamma^{i-1}}dyds\right]^{\gamma}.
 \eal
 \]
  Repeat this process and we arrive at:
\[\al\left(\int_{P(\sigma_{i+1})}h^{2\gamma^{i+1}}dyds\right)^{\frac{1}{\gamma^{i+1}}}&\leq
\left(2c_{16}\right)^{\sum\gamma^{-j}}c_{19}^{\sum(j+1)\gamma^{-j+1}}\gamma^{4\sum(j-1)\gamma^{-j+1}}\times\\
&\hspace{1cm}\left(K_b^4(C_{1,4,1})+\|V\|_{L^{\frac{10}{3}}(P_{1,4,1})}^4+1\right)^{\sum\gamma^{-j+1}}\int_{P_{1,4,1}}h^2dyds.\eal\]
Note the sums in the exponents are all from $j=1$ to $j=i+1$. Let
$i\rightarrow\infty$. All the exponent series converge.  In
particular, the series in the exponent for
$\left(K_b^4(C_{1,4,1})+\|V\|_{L^{\frac{10}{3}}(P_{1,4,1})}^4+1\right)$
converges to $\frac{5}{2}$. Note also that
$\sigma_i\rightarrow\frac{3}{4}$. Therefore, we arrive at:
\begin{equation}\label{moserbeforeVomegar}\sup_{P_{2,3,1}}\left(\omega_r^2+\omega_z^2\right)\leq
c_{20}\left(K_b^4(C_{1,4,1})+\|V\|_{L^{\frac{10}{3}}(P_{1,4,1})}^4+1\right)^{\frac{5}{2}}\left(\int_{P_{1,4,1}}\omega_r^2dyds+\int_{P_{1,4,1}}\omega_z^2dyds\right).\end{equation}

It is time to note how $\|V\|_{L^{\frac{10}{3}}(P_{1,4,1})}^4$ is
controlled. Recall:
\[V=\left[\begin{array}{cc}
\frac{\partial v_r}{\partial r}-\frac{1}{r^2} & \frac{\partial v_z}{\partial r}\\
\ &\ \\
\frac{\partial v_r}{\partial z} & \frac{\partial v_z}{\partial z}
\end{array}\right].\]
Applying Proposition \ref{lemmawith10/3} with $P_{1,4,1}$ being the domain on
the left, $P_{\frac{1}{2},\frac{16}{3},1}$ being the domain on the
right we can deduce that:
\begin{equation}\label{Vbeforeestimate}\|V\|_{L^{\frac{10}{3}}(P_{1,4,1})}\leq
c_{21}\left(\|\omega_{\theta}\|_{L^{\frac{10}{3}}(P_{\frac{1}{2},\frac{16}{3},1})}+\|v\|_{L^{\frac{10}{3}}(P_{\frac{1}{2},\frac{16}{3},1})}+1\right)\end{equation}
Even though at this point we already know that $V$ is a priori
bounded by standard energy estimates and our pointwise bound on
$\omega_{\theta}$, we use the method in Section 3 to prove a bound
for
$\|\omega_{\theta}\|_{L^{\frac{10}{3}}(P_{\frac{1}{2},\frac{16}{3},1})}$.
 This allows for better control of $\|V\|_{L^{\frac{10}{3}}(P_{1,4,1})}$.  The argument amounts to running Moser's iteration only once.
 Recall:\[\Omega=\frac{\omega_{\theta}}{r}\] and that in Section 3
 we defined a constant $\Lambda$ and functions:
\[\begin{array}{cc}\Omegabar_+=\left\{\begin{array}{cc}\Omega+\Lambda&\Omega\geq 0,\\
\Lambda & \Omega<0,\end{array}\right. & \Omegabar_-=\left\{\begin{array}{cc}-\Omega+\Lambda&\Omega\leq 0,\\
\Lambda & \Omega>0.\end{array}\right.\end{array}.\] We will utilize
estimate (\ref{iteration1}) to control the$L^{\frac{10}{3}}$ norm
of $\omega_{\theta}$, but first we must manipulate the domains
that appear in the inequality to fit our current setting. We
recall (\ref{iteration1}) from Section 3:
\[\al \int_{P(\sigma_2)}\Omegabar_+^{2q\gamma}dyds&\leq c_{22}\left(\frac{c_{23}q^2}{\tau^4}\left(K_b^4(C_{1,4,1})+\Lambda^2+1\right)\int_{P(\sigma_1)}\Omegabar_+^{2q}dyds\right)^{\gamma},
\eal\] where $P(\sigma_i)=P_{5-4\sigma_i,4\sigma_i,1}$,
$\tau=\sigma_1-\sigma_2$, and $\gamma=1=\frac{2}{n}$.
 We replace this $P(\sigma_i)$ with $P(\sigma_i)=P_{\frac{1}{4}(5-4\sigma_i),\frac{64}{9}\sigma_i,1}$.  The argument over
this domain would be identical to that in Section 3, with
$\Lambda=\|v_{\theta}\|_{L^{\infty}(P_{\frac{1}{4},\frac{64}{9},1})}\leq
4\|rv_{0,\theta}\|_{L^{\infty}(\reals^3)}$, up until the point
where we derive (\ref{iteration1}). We recall the condition on $q$
is $q\geq1$ and the condition on $\sigma_1,\sigma_2$ here, in this
setting, would be $\frac{265}{512}\leq\sigma_2<\sigma_1\leq1$.
Also note that $\gamma=1+\frac{2}{n},\ n=3$ and so
$\gamma=\frac{5}{3}$.  We choose $q=1,\ \sigma_1=1,\
\sigma_2=\frac{3}{4}$ to get:
\[\al \int_{P_{\frac{1}{2},\frac{16}{3},1}}\Omegabar_+^{\frac{10}{3}}dyds&\leq c_{24}\left(\left(K_b^4(P_{\frac{1}{4},\frac{64}{9},1})+\Lambda^2+1\right)\int_{P_{\frac{1}{4},\frac{64}{9},1}}\Omegabar_+^2dyds\right)^{\frac{5}{3}}.
\eal\] Similarly we can also get:
\[\al \int_{P_{\frac{1}{2},\frac{16}{3},1}}\Omegabar_-^{\frac{10}{3}}dyds&\leq c_{24}\left(\left(K_b^4(C_{\frac{1}{4},\frac{64}{9},1})+\Lambda^2+1\right)\int_{P_{\frac{1}{4},\frac{64}{9},1}}\Omegabar_-^2dyds\right)^{\frac{5}{3}}.
\eal\] Taking the $\frac{3}{10}$ power of both sides we derive:
\[\al \|\Omegabar_+\|_{L^{\frac{10}{3}}(P_{\frac{1}{2},\frac{16}{3},1})}&\leq c_{25}\left(K_b^4(C_{\frac{1}{4},\frac{64}{9},1})+\Lambda^2+1\right)^{\frac{1}{2}}\|\Omegabar_+\|_{L^2(P_{\frac{1}{4},\frac{64}{9},1})},
\eal\] and
\[\al \|\Omegabar_-\|_{L^{\frac{10}{3}}(P_{\frac{1}{2},\frac{16}{3},1})}&\leq c_{25}\left(K_b^4(C_{\frac{1}{4},\frac{64}{9},1})+\Lambda^2+1\right)^{\frac{1}{2}}\|\Omegabar_-\|_{L^2(P_{\frac{1}{4},\frac{64}{9},1})}.
\eal\] We can combine the above two estimates to get:
\[\left\|\frac{\omega_{\theta}}{r}\right\|_{L^{\frac{10}{3}}(P_{\frac{1}{2},\frac{16}{3},1})}\leq
c_{26}\left(K_b^4(C_{\frac{1}{4},\frac{64}{9},1})+\Lambda^2+1\right)^{\frac{1}{2}}\left\|\frac{\omega_{\theta}}{r}\right\|_{L^2(P_{\frac{1}{4},\frac{64}{9},1})}.\]
We note $r$ is bounded between two positive constants on the left
and on the right, to arrive at:
\[\|\omega_{\theta}\|_{L^{\frac{10}{3}}(P_{\frac{1}{2},\frac{16}{3},1})}\leq
c_{27}\left(K_b^4(C_{\frac{1}{4},\frac{64}{9},1})+\Lambda^2+1\right)^{\frac{1}{2}}\|\omega_{\theta}\|_{L^2(P_{\frac{1}{4},\frac{64}{9},1})}.\]
Apply this to (\ref{Vbeforeestimate}):
\[\|V\|_{L^{\frac{10}{3}}(P_{1,4,1})}\leq c_{28}\left(\left(K_b^4(C_{\frac{1}{4},\frac{64}{9},1})+\Lambda^2+1\right)^{\frac{1}{2}}\|\omega_{\theta}\|_{L^2(P_{\frac{1}{4},\frac{64}{9},1})}+\|v\|_{L^2(P_{\frac{1}{2},\frac{16}{3},1})}+1\right).\]
Thus,
\[\|V\|_{L^{\frac{10}{3}}(P_{1,4,1})}^4\leq c_{29}\left(\left(K_b^4(C_{\frac{1}{4},\frac{64}{9},1})+\|rv_{0,\theta}\|_{L^{\infty}(\reals^3)}+1\right)^2\|\omega_{\theta}\|_{L^2(P_{\frac{1}{4},\frac{64}{9},1})}^4+\|v\|_{L^2(P_{\frac{1}{2},\frac{16}{3},1})}^4+1\right),\]
utilizing $\Lambda\leq 4\|v_{0,\theta}\|_{L^{\infty}(\reals^3)}$.
Apply this to (\ref{moserbeforeVomegar}), we get:
\[\sup_{P_{2,3,1}}\left(\omega_r^2+\omega_z^2\right)\leq A\left(\int_{P_{1,4,1}}\omega_r^2dyds+\int_{P_{1,4,1}}\omega_z^2dyds\right),\]
where A is the constant defined as:
\[A=c_{30}\left(K_b^4(C_{\frac{1}{10},10,1})+\left(K_b^4(C_{\frac{1}{10},10,1})+\|rv_{0,\theta}\|_{L^{\infty}(\reals^3)}+1\right)\|\omega_{\theta}\|_{L^2(P_{\frac{1}{10},10,1})}^2+\|v\|_{L^2(P_{\frac{1}{10},10,1})}^2+1\right)^5.\]
The domain is enlarged proportionally to make the right hand side more uniform.

\vspace{.5cm}
\textbf{Re-scaling:} Recall our "tilde" notation and that what has
actually been shown to this point is:
\begin{equation}\label{beforerescalingr}\sup_{P_{2,3,1}}\left(\~{\omega}_r^2+\~{\omega}_z^2\right)(\~{x},\~{t})\leq
\~{A}\left(\int_{P_{1,4,1}}\~{\omega}_r^2d\~{x}d\~{t}+\int_{P_{1,4,1}}\~{\omega}_z^2d\~{x}d\~{t}\right),\end{equation}
where $\~{x}=\frac{x}{k},\ \~{t}=\frac{t}{k^2},\
\~{\omega}_r(\~{x},\~{t})=k^2\omega_r(k\~{x},k^2\~{t})$, $
\~{\omega}_z(\~{x},\~{t})=k^2\omega_z(k\~{x},k^2\~{t})$, and\[\~{A}=c_{30}\left(K_{\~{b}}^4(C_{\frac{1}{10},10,1})+\left(K_{\~{b}}^4(C_{\frac{1}{10},10,1})+\|\~{r}\~{v}_{0,\theta}\|_{L^{\infty}(\reals^3)}+1\right)\|\~{\omega}_{\theta}\|_{L^2(P_{\frac{1}{10},10,1})}^2+\|\~{v}\|_{L^2(P_{\frac{1}{10},10,1})}^2+1\right)^5.\] From the
scaling in Section 2:
\[K_{\~{b}}(C_{\frac{1}{10},10,1})=\|\~{b}(\~{x},\~{t})\|_{L^{\infty}(-1,0;(C_{\frac{1}{10},10,1}))}=\frac{1}{k^{\frac{1}{2}}}\|b\|_{L^{\infty}(-k^2,0;L^2(C_{\frac{1}{10},10,k}))},\]
\[\|\~{v}(\~{x},\~{t})\|_{L^2(P_{\frac{1}{10},10,1})}\frac{1}{k^{\frac{3}{2}}}=\|v(x,t)\|_{L^2(P_{\frac{1}{10},10,k})},\]
and
\[\|\~{\omega}(\~{x},\~{t})\|_{L^2(P_{\frac{1}{10},10,1})}=\frac{1}{k^{\frac{1}{2}}}\|\omega(x,t)\|_{L^2(P_{\frac{1}{10},10,k})}.\]
Also $\|rv_{0,\theta}\|_{L^{\infty}(\reals^3)}$ is scaling
invariant.  Finally, $\~{A}$
scales in the following way:
\[\al\~{A}&=c_{30}\left(K_{\~{b}}^4(C_{\frac{1}{10},10,1})+\left(K_{\~{b}}^4(C_{\frac{1}{10},10,1})+\|\~{r}\~{v}_{0,\theta}\|_{L^{\infty}(\reals^3)}+1\right)\|\~{\omega}_{\theta}\|_{L^2(P_{\frac{1}{10},10,1})}^2+\|\~{v}\|_{L^2(P_{\frac{1}{10},10,1})}^2+1\right)^5\\
&=\frac{c_{31}}{k^{15}}\left[\left(K_{b}^4(C_{\frac{1}{10},10,k})+k^2\|rv_{0,\theta}\|_{L^{\infty}(\reals^3)}+k^2\right)\|\omega_{\theta}\|_{L^2(P_{\frac{1}{10},10,k})}^2\right.\\
&\hspace{9cm}\left.+kK_{b}^4(C_{\frac{1}{10},10,k})+\|v\|_{L^2(P_{\frac{1}{10},10,k})}^2+k^3\right]^5.\eal\]
Apply all of this to
(\ref{beforerescalingr}) to achieve:
\[\al\sup_{P_{2,3,k}}k^4&\left(\omega_r^2(x,t)+\omega_z^2(x,t)\right)\\
&\leq\frac{c_{31}}{k^{16}}\left[\left(K_{b}^4(C_{\frac{1}{10},10,k})
+k^2\|rv_{0,\theta}\|_{L^{\infty}(\reals^3)}+k^2\right)\|\omega_{\theta}\|_{L^2(P_{\frac{1}{10},10,k})}^2\right.\\
&\hspace{2cm}\left.kK_{b}^4(C_{\frac{1}{10},10,k})+\|v\|_{L^2(P_{\frac{1}{10},10,k})}^2+k^3\right]^5\left(\|\omega_r\|_{L^2(P_{\frac{1}{10},10,k})}^2+\|\omega_z\|_{L^2(P_{\frac{1}{10},10,k})}^2\right).\eal\]
Therefore,
\[\al\|&\omega_r\|_{L^{\infty}(P_{2,3,k})}+\|\omega_z\|_{L^{\infty}(P_{2,3,k})}\\
&\leq\frac{c_{32}}{k^{10}}\left[\left(K_{b}^4(C_{\frac{1}{10},10,k})
+k^2\|rv_{0,\theta}\|_{L^{\infty}(\reals^3)}+k^2\right)\|\omega_{\theta}\|_{L^2(P_{\frac{1}{10},10,k})}^2\right.\\
&\hspace{3cm}\left.+kK_{b}^4(C_{\frac{1}{10},10,k})+\|v\|_{L^2(P_{\frac{1}{10},10,k})}^2+k^3\right]^{\frac{5}{2}}\left(\|\omega_r\|_{L^2(P_{\frac{1}{10},10,k})}+\|\omega_z\|_{L^2(P_{\frac{1}{10},10,k})}\right).\eal\]
This proves (ii) of Theorem \ref{ourtheorem}.

\begin{acknowledgement} We thank Professor T. P.
Tsai for sending us \cite{CSTY1} before its publication, and for his
useful suggestions.
\end{acknowledgement}
\end{section}

\end{document}